\documentclass[11pt]{article}

\usepackage[margin=1.1in]{geometry}
\usepackage{amsmath,amssymb,amsthm,mathtools}
\usepackage{enumitem}
\usepackage[colorlinks=true,linkcolor=blue,citecolor=blue,urlcolor=blue]{hyperref}
\usepackage{array}
\usepackage{graphicx}

\newtheorem{theorem}{Theorem}[section]
\newtheorem{proposition}[theorem]{Proposition}
\newtheorem{lemma}[theorem]{Lemma}
\newtheorem{corollary}[theorem]{Corollary}
\newtheorem{fact}[theorem]{Fact}
\theoremstyle{definition}
\newtheorem{definition}[theorem]{Definition}
\newtheorem{example}[theorem]{Example}
\theoremstyle{remark}
\newtheorem{remark}[theorem]{Remark}
\newtheorem{notation}[theorem]{Notation}

\newcommand{\Fm}{\mathsf{Form}}
\newcommand{\PL}{\mathsf{Prop}}
\newcommand{\SysK}{\mathsf{K}}
\newcommand{\SysT}{\mathsf{T}}
\newcommand{\SysF}{\mathsf{F}}
\newcommand{\SysS}{\mathsf{S}}
\newcommand{\vK}{\vdash_{\SysK}}
\newcommand{\vT}{\vdash_{\SysT}}
\newcommand{\vF}{\vdash_{\SysF}}

\newcommand{\nvF}{\nvdash_{\SysF}}
\newcommand{\pow}{\wp}

\newcommand{\open}{\mathrel{\lhd}}       
\newcommand{\nepo}{\mathrel{\rhd}}       
\newcommand{\preref}{\mathrel{\leq_{\open}}}        
\newcommand{\postref}{\mathrel{\leq_{\nepo}}}       
\newcommand{\strref}{\mathrel{\preccurlyeq_{\open}}}

\newcommand{\true}[1]{\lVert #1 \rVert}

\DeclareMathOperator{\FP}{FP}
\DeclareMathOperator{\SF}{Sub}

\newcommand{\rn}[1]{(\textsc{#1})}

\newcommand{\Fo}{\mathrm{Form}}
\newcommand{\den}[1]{\lVert #1\rVert}          
\newcommand{\openkern}{\mathrel{\lhd\!\rhd}}    
\newcommand{\lhds}{\openkern}                   
\newcommand{\pre}{\preref}                      
\newcommand{\post}{\postref}                    
\newcommand{\prep}{\strref}                     
\newcommand{\shook}{\mathbin{\rightarrow_{\!s}}}
\newcommand{\tI}{^{\mathrm{I}}}                 
\newcommand{\tO}{^{\mathrm{O}}}                 
\newcommand{\OS}{\mathsf{OS4}}
\newcommand{\ITB}{\mathsf{ITB}}

\newcommand{\sysF}{\mathsf{F}}
\newcommand{\Lb}{\mathcal{L}_{\Box}}
\newcommand{\Lbd}{\mathcal{L}_{\Box\Diamond}}
\newcommand{\Thm}{\mathrm{Thm}}
\newcommand{\Th}{\mathrm{Th}}
\newcommand{\dwn}{{\downarrow_{\vdash}}}

\title{Fundamental Propositional Logic with Preconditional:\\
Strong Completeness, Finite Model Property,\\ and Modal Translations}
\author{[Zhicheng Chen]\thanks{[Institute of Philosophy, Chinese Academy of Sciences]}}
\date{\today}

\begin{document}
\maketitle

\begin{abstract}
Fundamental logic (Holliday 2023) is a non-classical logic based only on the
introduction and elimination rules for conjunction, disjunction, and negation
in a Fitch-style natural deduction system, while a preconditional (Holliday 2025) is a binary operation on a bounded lattice satisfying five natural axioms and subsuming Heyting implication, the Sasaki hook on ortholattices, and Lewis–Stalnaker-style conditionals satisfying flattening. We combine the
two by giving a consequence-relation presentation $\SysK$ whose algebras are
exactly Holliday's bounded lattices with a preconditional, and then studying
two natural extensions, $\SysT$ and $\SysF$, the latter being fundamental
propositional logic with a preconditional. For $\SysT$ and $\SysF$, we prove strong completeness with
respect to a purely relational semantics, using a canonical model whose
points are pairs of theories, and establish the finite model property and
hence decidability. Finally, following Holliday and Massas (2026), we adapt
their GMT- and Goldblatt-style embeddings to fundamental logic with a
preconditional. The resulting translations are full and faithful into
ortho-$\mathsf{S4}$ and intuitionistic $\mathsf{KTB}$, respectively. The new
conditional clauses send the preconditional to a boxed Sasaki hook on the
former side and to a strict intuitionistic conditional on the latter whose
classical $\mathsf{KTB}$ reading is equivalent to the Goldblatt translation of
the Sasaki hook. The frame constructions follow the reduct-and-companion
pattern of Holliday and Massas; the essential additional ingredient is the
semantic transfer calculation for the preconditional.
\end{abstract}

\tableofcontents

\section{Introduction}\label{sec:intro}

Fundamental logic, introduced in \cite{Holliday2023}, is the non-classical
logic based only on the introduction and elimination rules for conjunction,
disjunction, and negation (and the quantifiers) in a Fitch-style natural
deduction system. Adding to fundamental logic the rule of Reductio ad
Absurdum yields a proof system for orthologic \cite{Goldblatt1974}, while
adding the rule that Fitch called Reiteration yields a proof system for the
$\to$-free fragment of intuitionistic logic; adding both yields classical
logic. Semantically, fundamental logic is the logic of \emph{fundamental
lattices}---bounded lattices equipped with a weak
pseudocomplementation---and, at the level of frames, the logic of reflexive,
pseudo-symmetric \emph{relational frames} $(X,\open)$, where $\open$ is a
relation of ``openness'' between states
\cite{Holliday2022,Holliday2023,HM2026}.

A natural question, already raised in \cite{Holliday2023}, concerns the
addition of a conditional connective to fundamental logic. In
\cite{Holliday2025}, Holliday isolated the class of \emph{preconditionals}:
binary operations $\to$ on a bounded lattice satisfying five simple axioms,
each intuitively valid for indicative and counterfactual conditionals in
natural language. Preconditionals subsume Heyting implication, the Sasaki
hook on ortholattices, and Lewis--Stalnaker-style conditionals
\cite{Stalnaker1968,Lewis1973} satisfying the flattening axiom, and every
bounded lattice with a preconditional is
representable over a relational frame $(X,\open)$, where the conditional is
interpreted by
\[
A \to_{\open} B \;=\; \{\, x \in X \mid \forall y \open x\, (y \in A
\Rightarrow \exists z \nepo y\colon z \in A \cap B) \,\},
\]
generalizing the usual semantics of intuitionistic implication and of the
Sasaki hook.

A closely related recent development is Chen's study of fundamental
propositional logic with strict implication \cite{Chen2026}. Chen takes strict
implication as a primitive connective and shows that, once it is added, the
consequence relation over pseudo-reflexive, pseudo-symmetric frames comes
apart from the consequence relation over reflexive, pseudo-symmetric frames,
with separate axiomatizations and completeness theorems. The present paper is
complementary: rather than taking strict implication as primitive, we add the
broader class of Holliday preconditionals to fundamental logic. Strict
implication is recovered in two ways: semantically, when the openness relation
is specialized to the reflexive and transitive frames underlying
intuitionistic implication; and modally, as the image of the preconditional
under the Goldblatt-style translation.

Meanwhile, building on the modal treatment of fundamental logic in
\cite{Holliday2024}, Holliday and Massas \cite{HM2026} studied fundamental logic
``through the lens of modality'': they proved that the
G\"{o}del--McKinsey--Tarski (GMT) translation of intuitionistic logic into
classical $\mathsf{S4}$ \cite{Godel1933,McKinseyTarski1948} is a full and
faithful embedding of fundamental logic
into \emph{ortho}-$\mathsf{S4}$ ($\mathsf{S4}$ on an orthological base), and
that the Goldblatt translation of orthologic into classical $\mathsf{KTB}$ is
a full and faithful embedding of fundamental logic into an
\emph{intuitionistic} $\mathsf{KTB}$, in the style of Fischer Servi
\cite{FS1984}. Their proofs are driven by frame transformations between
fundamental frames and birelational frames (the \emph{fundamental reduct} of
a modal frame and, conversely, the \emph{modal companions} of a fundamental
frame).

Against this background, the present paper extends fundamental logic with Holliday's preconditionals and studies the resulting logic within the modal framework developed by Holliday and Massas.\footnote{This is the work
referred to in \cite[\S~1]{HM2026} as extending their Theorems 1.1 and 1.2 to
fundamental logic with a preconditional.} Our contributions are as follows.

\begin{enumerate}[label=(\arabic*),leftmargin=2em]
\item In Section~\ref{sec:proof}, we axiomatize, as consequence relations
$\vdash\, \subseteq \pow(\Fm)\times\Fm$, three systems $\SysK \subseteq \SysT
\subseteq \SysF$ in the language with $\wedge$, $\vee$, $\to$, $\bot$, and $\top$,
where $\neg\varphi := \varphi\to\bot$. The system $\SysK$ is a
syntactic consequence presentation of Holliday's preconditional base
\cite{Holliday2025}; $\SysT$ adds the conditional-identity axiom
$\vdash\alpha\to\alpha$ and \emph{ex contradictione}; and $\SysF$ adds
double negation
introduction. We show that the $\{\wedge,\vee,\neg\}$-fragment of $\SysF$ is
a conservative extension of fundamental logic, so that $\SysF$ deserves the
name \emph{fundamental logic with a preconditional}.
\item In Section~\ref{sec:algebra}, we give algebraic semantics: $\SysK$
(resp.~$\SysT$, $\SysF$) is sound and complete with respect to bounded
lattices with a preconditional (resp.~with conditional identity and
semicomplementation; resp.~further with double negation inflation).
$\SysF$-algebras are exactly fundamental lattices equipped with a
preconditional satisfying conditional identity, and they include all Heyting
algebras and all ortholattices with the Sasaki hook.
\item In Section~\ref{sec:relational}, we give relational semantics over
frames $(X,\open)$ in the style of
\cite{Holliday2023,Holliday2025,HM2026} and prove soundness: $\SysT$
(resp.~$\SysF$) is sound with respect to reflexive (resp.~reflexive and
pseudo-symmetric) frames. In Section~\ref{sec:canonical}, we prove
\emph{strong} completeness by a purely relational canonical model
construction, whose points are pairs $(\Gamma,\Delta)$ of a theory and a
counter-theory, in the spirit of the filter-ideal pairs of
\cite{Holliday2023,HM2026} but formulated entirely syntactically. In
Section~\ref{sec:fmp}, a filtration-style refinement of the construction
yields the finite model property and decidability.
\item In Sections~\ref{sec:os4} and~\ref{sec:itb}, we extend the two main
translation theorems of \cite{HM2026} to the language with a conditional. For
the ortho-$\mathsf{S4}$ target, we use the GMT-style preconditional clause
\[
(\alpha\to\beta)^{I} \;=\; \Box\bigl(\alpha^{I}\shook\beta^{I}\bigr),
\]
where $\shook$ is the Sasaki hook. Thus this is an adaptation of the usual GMT
implication clause: the boxed material implication is replaced by the boxed
Sasaki hook. The resulting translation fully and faithfully embeds $\SysF$
into ortho-$\mathsf{S4}$. For the intuitionistic $\mathsf{KTB}$ target, the
Goldblatt-style translation uses
\[
(\alpha\to\beta)^{O} \;=\; \Box\bigl(\alpha^{O} \to
\Diamond(\alpha^{O}\wedge\beta^{O})\bigr),
\]
where the unadorned arrow on the right is the primitive intuitionistic
implication of the target language. This translation also fully and
faithfully embeds $\SysF$. The conceptual source of the displayed Goldblatt translation clause for the preconditional is again the Sasaki hook: on classical $\mathsf{KTB}$-models the
right-hand-side formula is equivalent to the Goldblatt translation of the Sasaki hook. Reading
the arrow intuitionistically therefore gives the natural Fischer--Servi
continuation of that Goldblatt image.

To prove these two embedding results, we use the reduct-and-companion framework of \cite{HM2026}. At the level of frames, \(F_1\) and \(F_2\) are their fundamental reduct constructions, while \(G_1\) and \(G_2\) are same-carrier analogues of the corresponding modal companions. We spell out the frame calculations needed for the preconditional explicitly. The additional work is the transfer of preconditional truth sets and the verification that the canonical model of Section~\ref{sec:canonical} satisfies the relevant factoring conditions (Definition~\ref{def:strong-ps}).

\end{enumerate}

Methodologically, we follow \cite{Holliday2023,Holliday2025,HM2026} in
developing the proof theory, algebra, and relational semantics in parallel,
and we adopt their notation and terminology throughout, so that results from
those papers can be cited directly.

Conceptually, the preconditional extension preserves the characteristic
mediating role of fundamental logic. Just as fundamental logic is a common
base for orthologic and the negation fragment of intuitionistic logic,
fundamental logic with a preconditional is meant to play the corresponding
role for richer conditional settings: on the one hand, orthologic equipped with the Sasaki hook; on the other, full intuitionistic logic equipped with Heyting implication.

\section{Proof systems}\label{sec:proof}

\subsection{Language}

Fix a countable set $\PL = \{p_0, p_1, \dots\}$ of propositional variables.
The language $\Fm$ is generated by the grammar
\[
\varphi \;::=\; p \mid \bot \mid \top \mid (\varphi \wedge \varphi) \mid (\varphi \vee
\varphi) \mid (\varphi \to \varphi),
\]
where $p \in \PL$. In the conditional-first presentation used below, the
bounds $\bot$ and $\top$ are both written as primitive constants, and we define
\[
\neg\alpha := \alpha \to \bot.
\]
Thus, in contrast to \cite{Holliday2023,HM2026}, negation is not primitive
but defined from the conditional, as in intuitionistic logic. This is a matter
of presentation: although $\bot$ is already interderivable with $\neg\top$ in
$\SysK$ (Proposition~\ref{prop:derived-K}), eliminating $\bot$ would require
taking $\neg$ as primitive instead, since negation is here defined from
$\bot$. In $\SysT$ and $\SysF$, Proposition~\ref{prop:derived-T} further gives
$\top \dashv\vdash \bot\to\bot$. We use
$\alpha,\beta,\varphi,\psi,\chi,\gamma,\dots$ for formulas and
$\Gamma,\Delta,\Theta,\dots$ for sets of formulas. We write
$\Gamma,\varphi\vdash\psi$ for $\Gamma\cup\{\varphi\}\vdash\psi$, and
$\alpha\dashv\vdash\beta$ for $\alpha\vdash\beta$ and $\beta\vdash\alpha$.

\subsection{The system \texorpdfstring{$\SysK$}{K}}

Our base system is a consequence relation whose conditional postulates are
exactly the preconditional axioms of \cite{Holliday2025}, in rule form.

\begin{definition}[System $\SysK$]\label{def:K}
Let $\vK$ be the smallest relation ${\vdash} \subseteq \pow(\Fm)\times\Fm$
satisfying the following conditions, for all $\Gamma,\Delta \subseteq \Fm$
and $\alpha,\beta,\varphi,\psi \in \Fm$:
\begin{itemize}[leftmargin=3.2em,itemsep=2pt]
\item[\rn{A}] $\Gamma, \varphi \vdash \varphi$;
\item[\rn{Cut}] if $\Gamma, \psi \vdash \varphi$ and $\Delta \vdash \psi$,
then $\Gamma \cup \Delta \vdash \varphi$;
\item[\rn{$\wedge$I}] $\{\varphi, \psi\} \vdash \varphi \wedge \psi$;
\item[\rn{$\wedge$E}] $\varphi \wedge \psi \vdash \varphi$ \ and \ $\varphi
\wedge \psi \vdash \psi$;
\item[\rn{$\vee$I}] $\varphi \vdash \varphi \vee \psi$ \ and \ $\psi \vdash
\varphi \vee \psi$;
\item[\rn{$\vee$E}] if $\varphi \vdash \alpha$ and $\psi \vdash \alpha$, then
$\varphi \vee \psi \vdash \alpha$;
\item[\rn{$\bot$E}] $\bot \vdash \varphi$;
\item[\rn{$\top$}] $\vdash \top$;
\item[\rn{$\top\to$E}] $\top \to \varphi \vdash \varphi$;
\item[\rn{Con}] if $\alpha \wedge \beta \vdash \psi$, then $\alpha \to \beta
\vdash \alpha \to \psi$;
\item[\rn{$\wedge{\vdash}\to$}] $\alpha \wedge \beta \vdash \alpha \to
\beta$;
\item[\rn{Imp}] if $\alpha \wedge \beta \vdash \varphi$, then $\varphi \to
(\alpha \to \beta) \vdash \varphi \wedge \alpha \to \beta$;
\item[\rn{Rep}] if $\alpha \dashv\vdash \alpha'$, then $\alpha \to \beta
\dashv\vdash \alpha' \to \beta$.
\end{itemize}
The rules \rn{A}, \rn{Cut}, \rn{$\wedge$I}, \rn{$\wedge$E}, \rn{$\vee$I},
\rn{$\vee$E}, \rn{$\bot$E}, and \rn{$\top$}---those not involving
$\to$---are called the \emph{basic rules}.
\end{definition}

Some comments on the rules are in order. \rn{Con} is a rule of
\emph{cautious right weakening}: under the supposition of $\alpha$, one may
replace the consequent $\beta$ by anything that follows from $\alpha$
together with $\beta$. \rn{$\wedge{\vdash}\to$} is \emph{conjunctive
sufficiency}. \rn{Imp} is a rule form of \emph{importation} for conditional
antecedents whose conjunction entails the outer antecedent; as we will see in
the proof of Theorem~\ref{thm:algebraic}, it corresponds exactly to axiom~5 of
preconditionals (the left-to-right half of the \emph{flattening} axiom of
\cite{Mandelkern2024}, see \cite[\S~2.5]{Holliday2025}). \rn{Rep} says that
the conditional is congruential in its antecedent; it is needed because,
unlike consequents (cf.~\rn{Con}), antecedents are not governed by any
monotonicity rule.

This definition is not meant to introduce a new algebraic base. On the
algebraic side, $\SysK$-algebras will be exactly bounded lattices equipped
with a preconditional in the sense of \cite{Holliday2025}. What is needed
here is the proof-theoretic counterpart: a consequence relation in which the
extensions $\SysT$ and $\SysF$, the canonical constructions, and the modal
translations can be formulated syntactically.

\begin{remark}[Monotonicity and finitarity]\label{rem:mon-fin}
The weakening rule
\begin{itemize}[leftmargin=3.2em]
\item[\rn{Mon}] if $\Gamma \vdash \varphi$, then $\Gamma \cup \Delta \vdash
\varphi$
\end{itemize}
is derivable from \rn{A} and \rn{Cut}. Moreover, $\vK$ is \emph{finitary}:
if $\Gamma \vK \varphi$, then $\Gamma_0 \vK \varphi$ for some finite
$\Gamma_0 \subseteq \Gamma$. Indeed, the relation
\[
{\vdash'} = \{(\Gamma,\varphi) \mid \Gamma_0 \vdash^{\mathrm{fin}} \varphi
\text{ for some finite } \Gamma_0 \subseteq \Gamma\},
\]
where $\vdash^{\mathrm{fin}}$ is generated by the rules of
Definition~\ref{def:K} restricted to finite premise sets, is easily checked
to be closed under all the rules, whence ${\vK} \subseteq {\vdash'}$. The
same applies to the systems $\SysT$ and $\SysF$ below. Consequently, by
\rn{$\wedge$I} and \rn{Cut}, for any of our systems $\SysS$:
\[
\Gamma \vdash_{\SysS} \varphi \quad\text{iff}\quad \gamma_1 \wedge \dots
\wedge \gamma_n \vdash_{\SysS} \varphi \text{ for some } \gamma_1, \dots,
\gamma_n \in \Gamma,
\]
where for $n = 0$ the left-hand side is read as $\top \vdash_{\SysS}
\varphi$ (equivalent to $\varnothing\vdash_{\SysS}\varphi$ by \rn{$\top$},
\rn{Mon}, and \rn{Cut}).
\end{remark}

\subsection{Derived rules of \texorpdfstring{$\SysK$}{K}}

\begin{proposition}[Derived rules of $\SysK$]\label{prop:derived-K}
The following are derivable in $\SysK$ (and hence in any extension of
$\SysK$):
\begin{itemize}[leftmargin=3.6em,itemsep=2pt]
\item[\rn{$\vdash\top\to$}] $\varphi \vdash \top \to \varphi$;
\item[\rn{$\to$Mon}] if $\beta \vdash \psi$, then $\alpha \to \beta \vdash
\alpha \to \psi$;
\item[\rn{$\bot\equiv\neg\top$}] $\bot \dashv\vdash \neg\top$;
\item[\rn{CP}] $\alpha \to \beta \vdash \alpha \to (\alpha \wedge \beta)$;
\item[\rn{Imp$^-$}] $\varphi \to (\varphi \wedge \alpha \to \beta) \vdash
\varphi \wedge \alpha \to \beta$;
\item[\rn{$\to\wedge$}] $\varphi \to (\alpha \wedge \beta) \vdash \varphi
\wedge \alpha \to \beta$;
\item[\rn{$\to\neg$}] $\alpha \to \neg\beta \vdash \neg(\alpha \wedge
\beta)$;
\item[\rn{$\neg$Ant}] if $\alpha \vdash \beta$, then $\neg\beta \vdash
\neg\alpha$;
\item[\rn{CAS}] if $\alpha \wedge \beta \vdash \varphi \to \psi$ and $\varphi
\wedge \psi \vdash \alpha$, then $\alpha \to \beta \vdash \alpha \wedge
\varphi \to \psi$;
\item[\rn{CAS$'$}] if $\alpha \wedge \beta \vdash \varphi \to \psi$ and
$\varphi \vdash \alpha$, then $\alpha \to \beta \vdash \varphi \to \psi$;
\item[\rn{$\wedge{\vdash}\neg$}] if $\alpha \wedge \beta \vdash \neg\varphi$,
then $\alpha \to \beta \vdash \neg(\alpha \wedge \varphi)$;
\item[\rn{$\wedge{\vdash}\bot$}] if $\alpha \wedge \beta \vdash \bot$, then
$\alpha \to \beta \vdash \neg\alpha$;
\item[\rn{$\vee$CAS}] if $\alpha \wedge \beta \vdash \neg\varphi \vee \psi$
and $\psi \vdash \alpha \wedge \varphi$, then $\alpha \to \beta \vdash \alpha
\wedge \varphi \to \psi$.
\end{itemize}
\end{proposition}

\begin{proof}
\rn{$\to\neg$}: Instantiate \rn{Imp} with $\alpha := \beta$, $\beta :=
\bot$, and outer antecedent $\alpha$: since $\beta \wedge \bot \vdash
\alpha$ by \rn{$\wedge$E}, \rn{$\bot$E}, and \rn{Cut}, we get $\alpha \to
(\beta \to \bot) \vdash \alpha \wedge \beta \to \bot$, i.e., $\alpha \to
\neg\beta \vdash \neg(\alpha \wedge \beta)$.

\rn{$\neg$Ant}: Suppose $\alpha \vdash \beta$. As in the previous item (with
the roles of $\alpha,\beta$ swapped), \rn{Imp} yields $\beta \to (\alpha \to
\bot) \vdash \beta \wedge \alpha \to \bot$. From $\alpha \vdash \beta$ we
get $\beta \wedge \alpha \dashv\vdash \alpha$ (using \rn{$\wedge$I},
\rn{$\wedge$E}, \rn{Cut}), so \rn{Rep} gives $\beta \wedge \alpha \to \bot
\dashv\vdash \alpha \to \bot$, whence $\beta \to (\alpha \to \bot) \vdash
\neg\alpha$. On the other hand, $\bot \vdash \alpha \to \bot$ by
\rn{$\bot$E}, so \rn{$\to$Mon} gives $\neg\beta = \beta \to \bot \vdash
\beta \to (\alpha \to \bot)$. Conclude by \rn{Cut}.

\rn{CAS}: From $\alpha \wedge \beta \vdash \varphi \to \psi$ and \rn{Con},
$\alpha \to \beta \vdash \alpha \to (\varphi \to \psi)$. From $\varphi
\wedge \psi \vdash \alpha$ and \rn{Imp}, $\alpha \to (\varphi \to \psi)
\vdash \alpha \wedge \varphi \to \psi$. Conclude by \rn{Cut}.

\rn{CAS$'$}: Since $\varphi \wedge \psi \vdash \varphi \vdash \alpha$,
\rn{CAS} gives $\alpha \to \beta \vdash \alpha \wedge \varphi \to \psi$.
From $\varphi \vdash \alpha$ we get $\alpha \wedge \varphi \dashv\vdash
\varphi$, so \rn{Rep} yields $\alpha \to \beta \vdash \varphi \to \psi$.
\rn{CAS$'$} follows.
\end{proof}

\subsection{The systems \texorpdfstring{$\SysT$}{T} and
\texorpdfstring{$\SysF$}{F}}

\begin{definition}[Systems $\SysT$ and $\SysF$]\label{def:TF}
Let $\vT$ be the smallest relation satisfying all conditions of
Definition~\ref{def:K}, together with:
\begin{itemize}[leftmargin=3.2em,itemsep=2pt]
\item[\rn{$\to$Id}] $\vdash \alpha \to \alpha$;
\item[\rn{Exp}] $\{\alpha, \neg\alpha\} \vdash \bot$.
\end{itemize}
Let $\vF$ be the smallest relation satisfying all conditions of $\SysT$
together with:
\begin{itemize}[leftmargin=3.2em]
\item[\rn{$\neg\neg$I}] $\alpha \vdash \neg\neg\alpha$.
\end{itemize}
\end{definition}

\rn{$\to$Id} is the proof-theoretic form of conditional identity (see Definition~\ref{def:algebras}). As
Proposition~\ref{prop:derived-T} shows, over $\SysK$ it yields
\rn{DT$_0$}, the ``easy'' half of the deduction theorem restricted to
categorical premises. Conversely, \rn{DT$_0$} yields \rn{$\to$Id} by applying it to $\chi\vdash\chi$. 
The ``hard'' contextual half (from $\Gamma,\alpha\vdash\beta$ to
$\Gamma\vdash\alpha\to\beta$) fails, as it must for any conditional
invalidating antecedent strengthening. \rn{Exp} makes the defined negation a
genuine (semi)complementation, corresponding to the elimination rule for
$\neg$ in fundamental logic, and \rn{$\neg\neg$I} corresponds to the
half of double negation determined by the introduction rule for $\neg$
\cite[\S~2]{Holliday2023}.

\begin{proposition}[Derived rules of $\SysT$]\label{prop:derived-T}
The following are derivable in $\SysT$ (hence also in $\SysF$):
\begin{itemize}[leftmargin=3.6em,itemsep=2pt]
\item[\rn{DT$_0$}] if $\alpha \vdash \beta$, then $\vdash \alpha \to
\beta$;
\item[\rn{$\top\equiv\bot\to\bot$}] $\top \dashv\vdash \bot\to\bot$;
\item[\rn{DT$\to$}] if $\varphi \vdash \alpha \to \beta$ and $\alpha \wedge
\beta \vdash \varphi$, then $\vdash (\varphi \wedge \alpha) \to \beta$;
\item[\rn{$\neg$MP}] $\{\alpha, \alpha \to \neg\beta, \beta\} \vdash \bot$;
\item[\rn{Con$\bot$}] if $\alpha \vdash \neg\alpha$, then $\alpha \vdash
\bot$.
\end{itemize}
\end{proposition}

\begin{proof}
\rn{DT$_0$}: if $\alpha\vdash\beta$, then \rn{$\to$Mon} gives
$\alpha\to\alpha\vdash\alpha\to\beta$; cut this with
$\vdash\alpha\to\alpha$ from \rn{$\to$Id}.
\rn{$\top\equiv\bot\to\bot$}: $\vdash\bot\to\bot$ is the instance of
\rn{$\to$Id} with $\alpha=\bot$; hence $\top\vdash\bot\to\bot$ by \rn{Mon}.
Conversely, $\bot\to\bot\vdash\top$ follows from \rn{$\top$} and \rn{Mon}.
\rn{DT$\to$}: from $\alpha \wedge \beta \vdash \varphi$ and \rn{Imp}, $\varphi
\to (\alpha \to \beta) \vdash (\varphi \wedge \alpha) \to \beta$; from $\varphi
\vdash \alpha \to \beta$ and \rn{DT$_0$}, $\vdash \varphi \to (\alpha \to
\beta)$; apply \rn{Cut}. \rn{$\neg$MP}: by \rn{$\to\neg$}, $\alpha \to
\neg\beta \vdash \neg(\alpha \wedge \beta)$; by \rn{$\wedge$I}, $\{\alpha,
\beta\} \vdash \alpha \wedge \beta$; by \rn{Exp}, $\{\alpha \wedge \beta,
\neg(\alpha \wedge \beta)\} \vdash \bot$; conclude by \rn{Cut} (twice).
\rn{Con$\bot$}: from $\alpha \vdash \neg\alpha$, \rn{A}, \rn{Exp}, and
\rn{Cut}.
\end{proof}


\subsection{Relation to fundamental logic}\label{subsec:rel-FL}

Recall from \cite[Def.~2.1]{Holliday2023} that an \emph{intro-elim logic} in
the language with $\wedge$, $\vee$, $\neg$ is a binary relation on formulas
satisfying reflexivity, transitivity, the lattice rules for $\wedge$ and
$\vee$, $\varphi \vdash \neg\neg\varphi$, $\varphi \wedge \neg\varphi \vdash
\psi$, and the contraposition rule (if $\varphi\vdash\psi$ then
$\neg\psi\vdash\neg\varphi$); fundamental logic $\vdash_{\mathcal{F}}$ is
the smallest intro-elim logic \cite[Prop.~3.8]{Holliday2023}.

\begin{proposition}\label{prop:F-extends-FL}
Regard the $\{\wedge,\vee,\neg\}$-language as included in $\Fm$ via
$\neg\alpha := \alpha\to\bot$. Then the restriction of $\vF$ to single
premises in this fragment is an intro-elim logic; hence
$\varphi\vdash_{\mathcal{F}}\psi$ implies $\varphi \vF \psi$.
\end{proposition}

\begin{proof}
Reflexivity and transitivity hold by \rn{A} and \rn{Cut}; the $\wedge$ and
$\vee$ rules are among the basic rules; $\varphi \vdash \neg\neg\varphi$ is
\rn{$\neg\neg$I}; $\varphi \wedge \neg\varphi \vdash \psi$ follows from
\rn{$\wedge$E}, \rn{Exp}, \rn{$\bot$E}, and \rn{Cut}; contraposition is
\rn{$\neg$Ant} (Proposition~\ref{prop:derived-K}).
\end{proof}

In fact the converse holds as well: $\vF$ is a \emph{conservative} extension
of fundamental logic. This will follow semantically from relational
soundness, since every fundamental frame in the sense of
\cite{Holliday2023,HM2026} interprets the conditional
(Corollary~\ref{cor:conservative}).

\section{Algebraic semantics}\label{sec:algebra}

We now match the systems of Section~\ref{sec:proof} with classes of bounded
lattices carrying a preconditional. Throughout, $L$ is a bounded lattice
with operations $\wedge, \vee$, bounds $0, 1$, and we trust that no
confusion will arise from using the same symbols as in $\Fm$.

\subsection{Preconditionals}

We recall the central definition of \cite{Holliday2025}.

\begin{definition}[{\cite[Def.~1]{Holliday2025}}]\label{def:precond}
A \emph{preconditional} on a bounded lattice $L$ is a binary operation $\to$
on $L$ satisfying, for all $a, b, c \in L$:
\begin{enumerate}[label=P\arabic*., leftmargin=2.6em]
\item $1 \to a \leq a$;
\item $a \wedge b \leq a \to b$;
\item $a \to b \leq a \to (a \wedge b)$;
\item $a \to (b \wedge c) \leq a \to b$;
\item $a \to ((a \wedge b) \to c) \leq (a \wedge b) \to c$.
\end{enumerate}
Given $\to$, we define $\neg a := a \to 0$.
\end{definition}

By \cite[Prop.~1]{Holliday2025}, the operation $\neg$ derived from a
preconditional is a \emph{precomplementation}: it is antitone ($a \leq b$
implies $\neg b \leq \neg a$) and satisfies $\neg 1 = 0$. The axioms P1--P5
are mutually independent \cite[Fact~1]{Holliday2025}.

The following arithmetic consists of short consequences of P1--P5 together
with the antitonicity of preconditional negation from
\cite[Prop.~1]{Holliday2025}; it will be used repeatedly.

\begin{lemma}\label{lem:precond-arith}
Let $\to$ be a preconditional on $L$ and $a, b, c, x \in L$.
\begin{enumerate}[label=(\arabic*)]
\item \textup{(Consequent monotonicity)} If $b \leq c$, then $a \to b \leq a
\to c$.
\item $a \to b = a \to (a \wedge b)$.
\item \textup{(Cautious antecedent strengthening)} If $a \wedge b \leq x$,
then $a \to b \leq (x \wedge a) \to b$.
\item \textup{(Importation)} If $a \wedge b \leq x$, then $x \to (a \to b)
\leq (x \wedge a) \to b$.
\item $a \to \neg b \leq \neg(a \wedge b)$, and $\neg b \leq \neg (a\wedge
b)$.
\end{enumerate}
\end{lemma}

\begin{proof}
We include the verification because items (1)--(4) are not isolated as
separate named results in \cite{Holliday2025}.
(1) If $b\le c$, then $a\wedge b=(a\wedge b)\wedge c$; P3 followed by P4
therefore gives $a\to b\le a\to c$.  Item (2) is P3 plus P4, since
$a\wedge b\le b$.

For (3), if $a\wedge b\le x$, then $a\wedge b=(a\wedge x)\wedge b\le
(a\wedge x)\to b$ by P2.  Using (2), (1), and then P5 gives
$a\to b\le a\to((a\wedge x)\to b)\le(a\wedge x)\to b$.  Item (4) follows
from (3) and (1):
$x\to(a\to b)\le x\to((x\wedge a)\to b)\le(x\wedge a)\to b$.

For (5), the inequality $\neg b\le\neg(a\wedge b)$ is the antitonicity of
preconditional negation from \cite[Prop.~1]{Holliday2025}. Then
$a\to\neg b\le a\to\neg(a\wedge b)=a\to((a\wedge b)\to0)\le(a\wedge b)\to0$
by (1) and P5.
\end{proof}

\subsection{\texorpdfstring{$\SysK$}{K}-, \texorpdfstring{$\SysT$}{T}-, and
\texorpdfstring{$\SysF$}{F}-algebras}

\begin{definition}\label{def:algebras}
A \emph{$\SysK$-algebra} is a pair $(L, \to)$ where $L$ is a bounded lattice
and $\to$ is a preconditional on $L$.

A \emph{$\SysT$-algebra} is a pair $(L, \to)$ where $L$ is a bounded lattice
and $\to$ is a preconditional on $L$ satisfying
\begin{enumerate}[label=(\roman*)]
\item \emph{conditional identity} \textup{(ID)}: $a \to a = 1$ for all
$a \in L$;
\item \emph{semicomplementation}: $a \wedge \neg a = 0$ for all $a \in L$.
\end{enumerate}

An \emph{$\SysF$-algebra} is a $\SysT$-algebra additionally satisfying
\begin{enumerate}[label=(\roman*),start=3]
\item \emph{double negation inflation}: $a \leq \neg\neg a$ for all $a \in
L$.
\end{enumerate}
\end{definition}

\begin{remark}
In the presence of the preconditional axioms, conditional identity
is equivalent to the condition that $a \leq b$ implies $a \to b = 1$: if $a
\leq b$, then $1 = a \to a = a \to (a \wedge b) = a \to b$ by
Lemma~\ref{lem:precond-arith}(2), since $a = a\wedge b$.
\end{remark}

\begin{remark}[$\SysF$-algebras and fundamental lattices]
\label{rem:F-fund}
Recall from \cite[Def.~2.3]{HM2026} that a \emph{fundamental lattice} is a
bounded lattice with a unary operation $\neg$ satisfying dual
self-adjointness ($a \leq \neg b$ implies $b \leq \neg a$) and
semicomplementation; equivalently, $\neg$ is antitone, doubly inflationary
($a \leq \neg\neg a$), and a semicomplementation---i.e., a \emph{weak
pseudocomplementation} in the sense of \cite{Holliday2023}. Since the
negation derived from a preconditional is automatically antitone, an
$\SysF$-algebra is exactly a bounded lattice $L$ with a preconditional $\to$
satisfying conditional identity such that $(L, \neg)$, with
$\neg a := a \to 0$, is a fundamental lattice. Thus $\SysF$-algebras are
fundamental lattices whose
weak pseudocomplementation is induced by a preconditional.
\end{remark}

\begin{example}[Heyting algebras]\label{ex:heyting}
By \cite[Prop.~3]{Holliday2025}, Heyting implications are exactly the
preconditionals satisfying modus ponens ($a \wedge (a \to b) \leq b$) and
weak monotonicity ($b \leq a \to b$). Every Heyting algebra is an
$\SysF$-algebra: conditional identity, semicomplementation, and
$a \leq \neg\neg a$ are standard. In particular, all pseudocomplemented
distributive lattices arise
as negation reducts of $\SysF$-algebras.
\end{example}

\begin{example}[Ortholattices with the Sasaki hook]\label{ex:sasaki}
Let $(O, \neg)$ be an ortholattice and $a \stackrel{s}{\to} b := \neg a \vee
(a \wedge b)$ the Sasaki hook. By \cite[Cor.~1]{Holliday2025},
$\stackrel{s}{\to}$ is a preconditional, and the negation it induces is the
orthocomplementation itself: $a \stackrel{s}{\to} 0 = \neg a$. Conditional
identity amounts to excluded middle, which holds in ortholattices, and
semicomplementation and double negation inflation are immediate. Hence
$(O,\stackrel{s}{\to})$ is an $\SysF$-algebra. Conversely, by
\cite[Prop.~4]{Holliday2025}, if $(L,\to)$ is a $\SysK$-algebra with $a
\wedge \neg a = 0$ and $\neg\neg a = a$ for all $a$, then $(L, \neg)$ is an
ortholattice and $\to$ is its Sasaki hook. Thus \emph{the $\SysF$-algebras
in which $\neg$ is involutive are exactly the ortholattices with the Sasaki
hook}.
\end{example}

\begin{remark}
In the taxonomy of \cite[Fig.~3]{Holliday2025}, $\SysT$-algebras lie between
``preconditionals with semicomplementation'' and ``proto-Heyting
implications'': conditional identity is derivable from weak monotonicity
\cite[Prop.~8.2]{Holliday2025} but not conversely (the Sasaki hook on a
non-Boolean ortholattice satisfies conditional identity but not weak
monotonicity).
Neither modus ponens nor normality is assumed; both fail in general in our
intended relational models (Section~\ref{sec:relational}), as they do for
the Sasaki hook.
\end{remark}

\subsection{Interpretations, soundness, and completeness}

\begin{definition}
Let $(L,\to)$ be a bounded lattice with a binary operation. A
\emph{valuation} is a map $\theta \colon \PL \to L$; it extends uniquely to
$\tilde\theta \colon \Fm \to L$ by $\tilde\theta(\bot) = 0$,
$\tilde\theta(\top)=1$, and the obvious recursive clauses for $\wedge$,
$\vee$, $\to$. Note that $\tilde\theta(\neg\alpha) =
\neg\tilde\theta(\alpha)$.

For a class $\mathsf{C}$ of such algebras, define $\Gamma
\vDash_{\mathsf{C}} \varphi$ iff for every $(L,\to) \in \mathsf{C}$ and
every valuation $\theta$, there are $\gamma_1, \dots, \gamma_n \in \Gamma$
(possibly $n = 0$) with $\tilde\theta(\gamma_1) \wedge \dots \wedge
\tilde\theta(\gamma_n) \leq \tilde\theta(\varphi)$, where the empty meet is
$1$.\footnote{This finitary definition matches the finitary consequence
relations of Section~\ref{sec:proof}; an equivalent formulation quantifies
over filters. The infinitary, local consequence relation is treated by the
relational semantics of Section~\ref{sec:relational}, with respect to which
we prove \emph{strong} completeness in Section~\ref{sec:canonical}.}
\end{definition}

\begin{theorem}[Algebraic soundness and completeness]\label{thm:algebraic}
Let $\SysS \in \{\SysK, \SysT, \SysF\}$ and let $\mathsf{C}_{\SysS}$ be the
class of $\SysS$-algebras. Then for all $\Gamma \subseteq \Fm$ and $\varphi
\in \Fm$:
\[
\Gamma \vdash_{\SysS} \varphi \quad\text{iff}\quad \Gamma
\vDash_{\mathsf{C}_{\SysS}} \varphi.
\]
\end{theorem}

\begin{proof}
\emph{Soundness.} By Remark~\ref{rem:mon-fin} it suffices to check that the
relation $\Gamma \vdash^{*} \varphi$ defined by ``$\tilde\theta(\gamma_1)
\wedge \dots \wedge \tilde\theta(\gamma_n) \leq \tilde\theta(\varphi)$ for
some $\gamma_i \in \Gamma$, for every algebra in $\mathsf{C}_{\SysS}$ and
valuation $\theta$'' is closed under the rules of $\SysS$. The basic rules
correspond to the bounded lattice laws, with \rn{$\top$} interpreted as the
top element $1$. For the conditional rules of $\SysK$:
\rn{$\top\to$E} is P1 (as $\tilde\theta(\top) = 1$);
\rn{$\wedge{\vdash}\to$} is P2; \rn{Con} holds by
Lemma~\ref{lem:precond-arith}(2) and (1): if $a \wedge b \leq c$ then $a \to
b = a \to (a \wedge b) \leq a \to c$; \rn{Imp} is
Lemma~\ref{lem:precond-arith}(4); and \rn{Rep} holds since $\to$ is an
operation on $L$ (equal arguments give equal values). For $\SysT$,
\rn{$\to$Id} is conditional identity and \rn{Exp} is
semicomplementation. For $\SysF$:
\rn{$\neg\neg$I} is double negation inflation.

\emph{Completeness.} Suppose $\Gamma \nvdash_{\SysS} \varphi$. Define
$\alpha \approx \beta$ iff $\alpha \dashv\vdash_{\SysS} \beta$; by \rn{A}
and \rn{Cut} this is an equivalence relation, and by \rn{$\wedge$I},
\rn{$\wedge$E}, \rn{$\vee$I}, \rn{$\vee$E}, \rn{$\to$Mon}, and \rn{Rep} it
is a congruence for $\wedge$, $\vee$, $\to$. Let $L_{\SysS}$ be the quotient
$\Fm/{\approx}$ with $[\alpha] \leq [\beta]$ iff $\alpha \vdash_{\SysS}
\beta$; this is a lattice with the induced operations, bounded by $0 =
[\bot]$ (by \rn{$\bot$E}) and $1 = [\top]$ (by \rn{$\top$}, \rn{Mon}, and
\rn{Cut}).

We check that $(L_{\SysS}, \to)$ is an $\SysS$-algebra. P1 is
\rn{$\top\to$E}, since $1=[\top]$; P2 is
\rn{$\wedge{\vdash}\to$}; P3 is \rn{CP}; P4 follows from \rn{$\to$Mon} and
\rn{$\wedge$E}; P5 follows from \rn{Imp} instantiated as in the proof of
\rn{Imp$^-$}: taking the outer antecedent $\alpha$ and inner antecedent
$\alpha \wedge \beta$, the side condition $(\alpha \wedge \beta) \wedge
\gamma \vdash \alpha$ holds by \rn{$\wedge$E}, giving $\alpha \to ((\alpha
\wedge \beta) \to \gamma) \vdash \alpha \wedge (\alpha \wedge \beta) \to
\gamma \dashv\vdash (\alpha \wedge \beta) \to \gamma$ by \rn{Rep}. For
$\SysS = \SysT$: conditional identity is
\rn{$\to$Id}, and semicomplementation is $\alpha \wedge \neg\alpha \vdash
\bot$, which follows from \rn{$\wedge$E}, \rn{Exp}, and \rn{Cut}. For $\SysS
= \SysF$: double negation inflation is \rn{$\neg\neg$I}.

Let $\theta(p) = [p]$; an easy induction gives $\tilde\theta(\alpha) =
[\alpha]$. If we had $[\gamma_1] \wedge \dots \wedge [\gamma_n] \leq
[\varphi]$ for some $\gamma_i \in \Gamma$, then $\gamma_1 \wedge \dots
\wedge \gamma_n \vdash_{\SysS} \varphi$, hence $\Gamma \vdash_{\SysS}
\varphi$ by \rn{$\wedge$I}, \rn{Mon}, and \rn{Cut} (or $\top \vdash \varphi$
and hence $\vdash \varphi$ if $n = 0$), a contradiction. So $\Gamma
\nvDash_{\mathsf{C}_{\SysS}} \varphi$.
\end{proof}

\begin{remark}
Theorem~\ref{thm:algebraic} for $\SysS = \SysK$ shows that, for single
premises, $\vK$ axiomatizes exactly the consequences valid in all bounded
lattices with a preconditional. Thus the algebraic content is precisely
Holliday's preconditional semantics \cite{Holliday2025}; the role of
$\SysK$ is to provide its consequence-relation counterpart for the later
syntactic and modal arguments.
\end{remark}

\section{Relational semantics}\label{sec:relational}

We now recall the relational semantics of
\cite{Holliday2023,Holliday2025,HM2026} and adapt it to our language. We
use the notation of \cite{HM2026} for the induced operators and of
\cite{Holliday2025} for the conditional operation. Following the convention
of \cite{HM2026}, the relational notation in this section is
\emph{predecessor-style}: a frame is written $(X,\open)$, and $y\open x$
means that $y$ is open to $x$, so $\Box_{\open}$ quantifies over the
states written on the left of $x$.  Thus, if one instead writes the usual
forward accessibility relation as $xRy$, then $xRy$ corresponds here to
$y\open x$.  The modal target frames in Sections~\ref{sec:os4} and
\ref{sec:itb} use the same predecessor-style convention for their
preorders: $y\le x$ means that $y$ is a refinement/predecessor of $x$,
and $\Box_{\le}$ quantifies over such $y$.  This is the converse of the
forward accessibility notation $xRy$ often used in modal semantics; it is
also the convention used by the operator definitions below.

\subsection{Frames, propositions, and the conditional operation}

\begin{definition}[{\cite[Def.~2.7]{HM2026}}]\label{def:frame}
A \emph{relational frame} is a pair $(X, \open)$ where $X$ is a nonempty set
and $\open$ is a binary relation on $X$. When $y \open x$, we say that $y$
is \emph{open to} $x$. A state $x \in X$ is \emph{absurd} if there is no $y
\open x$.
\end{definition}

\begin{notation}\label{not:operators}
Given a relational frame $(X, \open)$ and $U \subseteq X$, define
\[
\Box U = \{x \in X \mid \forall y \open x\colon y \in U\},
\qquad
\blacklozenge U = \{x \in X \mid \exists y\colon x \open y \text{ and } y
\in U\}.
\]
Thus $\Box$ quantifies universally over the states open to $x$, while
$\blacklozenge$ quantifies existentially over the states to which $x$ is
open. Further define, for $U, V \subseteq X$:
\begin{align*}
c_{\open}(U) &= \Box\blacklozenge U
= \{x \mid \forall y \open x\, \exists z\colon y \open z \text{ and } z \in
U\},\\
\neg_{\open} U &= \Box (X \setminus U)
= \{x \mid \forall y \open x\colon y \notin U\},\\
U \to_{\open} V &= \Box\bigl((X \setminus U) \cup \blacklozenge(U \cap
V)\bigr)
= \{x \mid \forall y \open x\, (y \in U \Rightarrow \exists z\colon y \open
z \text{ and } z \in U \cap V)\}.
\end{align*}
\end{notation}

The map $c_{\open}$ is the closure operator of
\cite[Def.~2.8]{HM2026},\footnote{In the notation of \cite{HM2026},
$c_{\open}(U) = \neg_{\open}\neg_{\nepo}U$, where $\nepo$ is the converse of
$\open$; and $\to_{\open}$ is the conditional operation of
\cite[\S~3]{Holliday2025}.} and a \emph{proposition} in $(X,\open)$ is a
fixpoint of $c_{\open}$. By \cite[Prop.~2.10]{HM2026}, the set of
propositions ordered by inclusion forms a complete lattice
$\chi_{\open}(X)$, in which meets are intersections, joins are given by
$\bigvee_{i} U_i = c_{\open}(\bigcup_i U_i)$, the top element is $X$, and
the least element $0 = c_{\open}(\varnothing) = \Box\varnothing$ is the set
of absurd states.

\begin{fact}[{\cite[\S~3, Fact~4]{Holliday2025}}]\label{fact:precond-frame}
For any relational frame $(X, \open)$, if $U, V \in \chi_{\open}(X)$, then
$U \to_{\open} V \in \chi_{\open}(X)$, and the operation $\to_{\open}$ is a
preconditional on the complete lattice $\chi_{\open}(X)$.
\end{fact}

\begin{definition}[{\cite[Lem.~2.11, Def.~2.12]{HM2026}}]\label{def:frame-conds}
Let $(X, \open)$ be a relational frame and $x, y \in X$. We say $y$
\emph{prerefines} $x$, written $y \preref x$, if for all $z \in X$, $z \open
y$ implies $z \open x$; equivalently, every proposition containing $x$
contains $y$. The frame is:
\begin{itemize}
\item \emph{reflexive} if $x \open x$ for all $x$;
\item \emph{pseudo-symmetric} if for all $x$ and all $y \open x$, there is
$z \open y$ with $z \preref x$.
\end{itemize}
An \emph{$\SysF$-frame} is a reflexive, pseudo-symmetric frame. Every
$\SysF$-frame is a \emph{fundamental frame} in the sense of
\cite[Def.~2.12]{HM2026}, since reflexivity implies pseudo-reflexivity;
conversely, the canonical fundamental frames of
\cite[Thm.~2.15]{HM2026} are reflexive, so restricting to $\SysF$-frames
does not change the induced logic.
\end{definition}

\subsection{Models and semantic consequence}

\begin{definition}\label{def:model}
A \emph{model} is a triple $\mathcal{M} = (X, \open, V)$ where $(X, \open)$
is a relational frame and $V \colon \PL \to \chi_{\open}(X)$ is an
\emph{admissible valuation}, i.e., every $V(p)$ is a proposition. The
\emph{truth set} $\true{\varphi}_{\mathcal{M}} \subseteq X$ of a formula is
defined recursively:
\begin{align*}
\true{p} &= V(p), \qquad \true{\top}=X, \qquad
\true{\bot} = c_{\open}(\varnothing) = \Box\varnothing,\\
\true{\alpha \wedge \beta} &= \true{\alpha} \cap \true{\beta}, \qquad
\true{\alpha \vee \beta} = c_{\open}(\true{\alpha} \cup \true{\beta}),\\
\true{\alpha \to \beta} &= \true{\alpha} \to_{\open} \true{\beta}.
\end{align*}
We write $\mathcal{M}, x \vDash \varphi$ for $x \in
\true{\varphi}_{\mathcal{M}}$, and $\mathcal{M}, x \vDash \Gamma$ if
$\mathcal{M}, x \vDash \gamma$ for all $\gamma \in \Gamma$. For a class
$\mathsf{C}$ of models,
\[
\Gamma \vDash_{\mathsf{C}} \varphi \quad\text{iff}\quad \text{for all }
\mathcal{M} \in \mathsf{C} \text{ and } x \in \mathcal{M}\colon\ 
\mathcal{M}, x \vDash \Gamma \Rightarrow \mathcal{M}, x \vDash \varphi.
\]
Let $\mathsf{C}_{\SysT}$ be the class of models over reflexive frames, and
$\mathsf{C}_{\SysF}$ the class of models over $\SysF$-frames.
\end{definition}

By Fact~\ref{fact:precond-frame} and \cite[Prop.~2.10]{HM2026}, an easy
induction shows that $\true{\varphi} \in \chi_{\open}(X)$ for every formula
$\varphi$; thus each model determines an interpretation into the dual
algebra of its frame. On every frame, the derived negation coincides with
the primitive negation of \cite{Holliday2023,HM2026}. Indeed, since
$\true{\bot}=\Box\varnothing$ and
$\blacklozenge(U\cap\Box\varnothing)=\varnothing$ for every $U\subseteq X$,
we have
\[
\true{\neg\alpha} = \true{\alpha} \to_{\open} \Box\varnothing =
\Box(X\setminus\true{\alpha}) = \neg_{\open}\true{\alpha}.
\]
Thus the $\{\wedge,\vee,\neg\}$-clauses of our semantics agree with theirs.

\begin{proposition}[Dual algebras]\label{prop:dual-alg}
Let $(X, \open)$ be a relational frame.
\begin{enumerate}[label=(\arabic*)]
\item $(\chi_{\open}(X), \to_{\open})$ is a bounded lattice with a
preconditional.
\item If $\open$ is reflexive, then $(\chi_{\open}(X), \to_{\open})$ is a
$\SysT$-algebra.
\item If $(X, \open)$ is an $\SysF$-frame, then $(\chi_{\open}(X),
\to_{\open})$ is an $\SysF$-algebra.
\end{enumerate}
\end{proposition}

\begin{proof}
Part (1) is precisely \cite[Fact~4]{Holliday2025}, together with the lattice
structure of propositions from \cite[Prop.~2.10]{HM2026}.  If $\open$ is
reflexive, then $0=\Box\varnothing=\varnothing$; reflexivity also gives
$U\to_{\open}U=X$ and $U\cap\neg_{\open}U=\varnothing$ for every proposition
$U$, so the dual algebra is a $\SysT$-algebra.  Finally, by
\cite[Lem.~2.13(2)]{HM2026}, pseudo-symmetry of $\open$ is equivalent to
$U\subseteq\neg_{\open}\neg_{\open}U$ for every proposition $U$.  Since the
negation derived from $\to_{\open}$ is $\neg_{\open}$ on every frame, this is
exactly double negation inflation.
\end{proof}

\begin{theorem}[Soundness]\label{thm:soundness}
For all $\Gamma \subseteq \Fm$ and $\varphi \in \Fm$:
\begin{enumerate}[label=(\arabic*)]
\item if $\Gamma \vT \varphi$, then $\Gamma \vDash_{\mathsf{C}_{\SysT}}
\varphi$;
\item if $\Gamma \vF \varphi$, then $\Gamma \vDash_{\mathsf{C}_{\SysF}}
\varphi$.
\end{enumerate}
\end{theorem}

\begin{proof}
Suppose $\Gamma \vT \varphi$ and let $\mathcal{M} = (X, \open, V) \in
\mathsf{C}_{\SysT}$ and $x \vDash \Gamma$. By Remark~\ref{rem:mon-fin},
$\gamma_1 \wedge \dots \wedge \gamma_n \vT \varphi$ for some $\gamma_i \in
\Gamma$ (reading the empty conjunction as $\top$). The valuation $V$
extends to the interpretation $\true{\cdot}$ into the $\SysT$-algebra
$(\chi_{\open}(X), \to_{\open})$ of Proposition~\ref{prop:dual-alg}
(note $\true{\bot} = \varnothing = 0$), so by the soundness half of
Theorem~\ref{thm:algebraic}, $\true{\gamma_1} \cap \dots \cap
\true{\gamma_n} \subseteq \true{\varphi}$. Since $x \in \true{\gamma_i}$
for each $i$ (and $\true{\top}=X$ if $n = 0$), we get $x
\in \true{\varphi}$. The argument for $\SysF$ is the same, using
Proposition~\ref{prop:dual-alg}(3).
\end{proof}

\begin{corollary}[Conservativity over fundamental
logic]\label{cor:conservative}
For all $\varphi, \psi$ in the $\{\wedge,\vee,\neg\}$-fragment
\textup{(}with $\neg\alpha := \alpha \to \bot$\textup{)}: $\varphi \vF
\psi$ iff $\varphi \vdash_{\mathcal{F}} \psi$ in fundamental logic.
\end{corollary}

\begin{proof}
Right to left is Proposition~\ref{prop:F-extends-FL}. Left to right:
suppose $\varphi \nvdash_{\mathcal{F}} \psi$. By the completeness of
fundamental logic with respect to reflexive, pseudo-symmetric frames with
the semantics of \cite{Holliday2023} (see \cite[Thm.~2.15]{HM2026}, whose
canonical frame is reflexive), there is a model over an $\SysF$-frame and a
state satisfying $\varphi$ but not $\psi$ under the
$\{\wedge,\vee,\neg\}$-semantics. As observed after
Definition~\ref{def:model}, our semantics for the fragment coincides with
that of \cite{Holliday2023,HM2026} (with
$\true{\neg\alpha} = \neg_{\open}\true{\alpha}$). Hence $\varphi
\nvDash_{\mathsf{C}_{\SysF}} \psi$, so $\varphi \nvF \psi$ by
Theorem~\ref{thm:soundness}.
\end{proof}

\begin{remark}\label{rem:failures}
Modus ponens, weak monotonicity, and normality all fail for $\to_{\open}$
on $\SysF$-frames in general; see \cite[Ex.~2]{Holliday2025} for
countermodels. This is as it should be: modus ponens fails for the Sasaki
hook on non-orthomodular ortholattices \cite{Mittelstaedt1972}, and weak
monotonicity fails for the
Sasaki hook on non-Boolean orthomodular lattices, both of which are
$\SysF$-algebras by Example~\ref{ex:sasaki} and hence, by the completeness
theorem below, must be respected by the semantics.
\end{remark}

\section{Canonical model and strong completeness}\label{sec:canonical}

In this section we prove the converses of Theorem~\ref{thm:soundness} by a
canonical model construction. The construction is a syntactic counterpart
of the filter-ideal frames used in
\cite[Thm.~4.30]{Holliday2023}, \cite[Thm.~2]{Holliday2025}, and
\cite[Thm.~2.15]{HM2026}: states are pairs $(\Gamma, \Delta)$ of a
\emph{theory} $\Gamma$ (corresponding to a filter) and a
\emph{counter-theory} $\Delta$ (corresponding to an ideal), and the
openness relation compares the theory of one state with the counter-theory
of another. The essential new ingredient, replacing the negation-driven
condition ``$a \in F$ implies $\neg a \in I$'' of \cite{HM2026}, is a
closure condition on $\Delta$ driven by the conditional
(Definition~\ref{def:FGamma}), which is the syntactic counterpart of the
compatibility condition on filter-ideal pairs in Holliday's representation
theorem for lattices with preconditionals \cite[Thm.~2]{Holliday2025}.

Throughout this section, $\vdash$ denotes either $\vT$ or $\vF$; each lemma
lists the rules it uses, so that it is clear which conclusions require
\rn{$\neg\neg$I}.

\subsection{Theories and counter-theories}

\begin{definition}\label{def:theories}
For $\Gamma \subseteq \Fm$, let $\Th(\Gamma) = \{\alpha \in \Fm \mid \Gamma
\vdash \alpha\}$, and write $\Th(\alpha)$ for $\Th(\{\alpha\})$. We say
$\Gamma$ is \emph{consistent} if $\Gamma \nvdash \bot$, and \emph{closed}
if $\Th(\Gamma) \subseteq \Gamma$. For $\Xi \subseteq \Fm$, let
\[
{\downarrow}\Xi = \{\beta \in \Fm \mid \beta \vdash \psi \text{ for some }
\psi \in \Xi\}, \qquad \neg[\Gamma] = \{\neg\gamma \mid \gamma \in
\Gamma\}.
\]
\end{definition}

\begin{definition}\label{def:FGamma}
For $\Gamma, \Delta \subseteq \Fm$, say that $\Delta$ is
\emph{$\to$-closed relative to $\Gamma$} if for all $\alpha \in \Gamma$ and
$\beta \in \Fm$: if $\alpha \wedge \beta \in \Delta$, then $\alpha \to
\beta \in \Delta$.
\end{definition}

\begin{definition}[Canonical states]\label{def:canonical-states}
Define:
\begin{align*}
W_d &= \bigl\{(\Gamma, \Delta) \;\big|\; \Gamma \text{ is consistent and
closed; } \bot \in \Delta; \Delta \text{ is closed under
$\vdash$-predecessors}\\
&\qquad\quad \text{(i.e., ${\downarrow}\Delta \subseteq \Delta$) and under
$\vee$; } \Delta \text{ is $\to$-closed relative to } \Gamma; \text{ and }
\Gamma \cap \Delta = \varnothing \bigr\};\\
W_a &= \bigl\{(\Gamma, {\downarrow}\neg[\Gamma]) \;\big|\; \Gamma \text{ is
consistent and closed} \bigr\};\\
W_b &= \bigl\{(\Th(\alpha), {\downarrow}\{\alpha \to \beta\}) \;\big|\;
\alpha, \beta \in \Fm,\ \alpha \nvdash \alpha \to \beta \bigr\}.
\end{align*}
For $x = (\Gamma, \Delta)$ we write $\Gamma_x = \Gamma$ and $\Delta_x =
\Delta$. Define the relation $\open$ on pairs by
\[
y \open x \quad\text{iff}\quad \Gamma_x \cap \Delta_y = \varnothing,
\]
matching the canonical openness relation of \cite[Thm.~2.15]{HM2026} and
\cite[Thm.~4.30]{Holliday2023} (with $\Gamma$ in place of the filter and
$\Delta$ in place of the ideal). For $W \subseteq W_d$, the \emph{canonical
frame over $W$} is $(W, \open)$ with $\open$ restricted to $W$, and we
define $|\cdot|_W \colon \Fm \to \pow(W)$ by
\[
|\varphi|_W = \{x \in W \mid \varphi \in \Gamma_x\}.
\]
\end{definition}

Points of $W_b$ will serve as generic witnesses: the state
$(\Th(\alpha), {\downarrow}\{\alpha \to \beta\})$ is the ``freest'' state
that accepts $\alpha$ while rejecting $\alpha \to \beta$ (and everything
that entails it).

\begin{lemma}[Least counter-theories]\label{lem:least-ct}
\textup{(Uses the basic rules, \rn{$\neg$Ant}, \rn{$\wedge{\vdash}\neg$},
\rn{Exp}.)} Let $\Gamma$ be consistent and closed. Then:
\begin{enumerate}[label=(\arabic*)]
\item $(\Gamma, {\downarrow}\neg[\Gamma]) \in W_d$; hence $W_a \subseteq
W_d$.
\item ${\downarrow}\neg[\Gamma] \subseteq \Delta$ for every $\Delta
\subseteq \Fm$ that contains $\bot$, is closed under
$\vdash$-predecessors, and is $\to$-closed relative to $\Gamma$. In
particular, $\Delta_x \supseteq {\downarrow}\neg[\Gamma_x]$ for every $x
\in W_d$.
\end{enumerate}
\end{lemma}

\begin{proof}
(1) Since $\Gamma$ is closed, $\top \in \Gamma$ by \rn{$\top$}, so $\bot \in
{\downarrow}\neg[\Gamma]$ as $\bot \vdash \neg\top$ by \rn{$\bot$E}.
Closure under $\vdash$-predecessors holds by \rn{Cut}. For closure under
$\vee$: if $\beta_i \vdash \neg\gamma_i$ with $\gamma_i \in \Gamma$
($i=1,2$), then $\gamma_1 \wedge \gamma_2 \in \Gamma$ (closure,
\rn{$\wedge$I}), and $\beta_i \vdash \neg\gamma_i \vdash \neg(\gamma_1
\wedge \gamma_2)$ by \rn{$\neg$Ant}, so $\beta_1 \vee \beta_2 \vdash
\neg(\gamma_1 \wedge \gamma_2)$ by \rn{$\vee$E}. For $\to$-closure relative
to $\Gamma$: suppose $\alpha \in \Gamma$ and $\alpha \wedge \beta \vdash
\neg\gamma$ with $\gamma \in \Gamma$. By \rn{$\wedge{\vdash}\neg$}, $\alpha
\to \beta \vdash \neg(\alpha \wedge \gamma)$, and $\alpha \wedge \gamma \in
\Gamma$, so $\alpha \to \beta \in {\downarrow}\neg[\Gamma]$. Finally, if
$\chi \in \Gamma \cap {\downarrow}\neg[\Gamma]$, say $\chi \vdash
\neg\gamma$ with $\gamma \in \Gamma$, then $\Gamma \vdash \gamma$ and
$\Gamma \vdash \neg\gamma$, so $\Gamma \vdash \bot$ by \rn{Exp} and
\rn{Cut}, contradicting consistency.

(2) Let $\Delta$ be as described and $\gamma \in \Gamma$. Then $\gamma
\wedge \bot \vdash \bot \in \Delta$, so $\gamma \wedge \bot \in \Delta$ by
predecessor closure, whence $\neg\gamma = \gamma \to \bot \in \Delta$ by
$\to$-closure relative to $\Gamma$. Thus $\neg[\Gamma] \subseteq \Delta$,
and predecessor closure gives ${\downarrow}\neg[\Gamma] \subseteq \Delta$.
\end{proof}

\begin{lemma}\label{lem:Wb-in-Wd}
\textup{(Uses the basic rules, \rn{CAS$'$}.)} $W_b \subseteq W_d$.
\end{lemma}

\begin{proof}
Let $x = (\Th(\alpha), {\downarrow}\{\alpha \to \beta\})$ with $\alpha
\nvdash \alpha \to \beta$. $\Th(\alpha)$ is closed by \rn{Cut} and
consistent, since $\alpha \vdash \bot$ would give $\alpha \vdash \alpha \to
\beta$ by \rn{$\bot$E} and \rn{Cut}. The set ${\downarrow}\{\alpha \to
\beta\}$ contains $\bot$ (by \rn{$\bot$E}), is closed under
$\vdash$-predecessors (by \rn{Cut}) and under $\vee$ (by \rn{$\vee$E}). For
$\to$-closure relative to $\Th(\alpha)$: suppose $\varphi \in \Th(\alpha)$
and $\varphi \wedge \psi \vdash \alpha \to \beta$. Since $\alpha \vdash
\varphi$, \rn{CAS$'$} yields $\varphi \to \psi \vdash \alpha \to \beta$,
i.e., $\varphi \to \psi \in {\downarrow}\{\alpha \to \beta\}$. Finally, if
$\chi \in \Th(\alpha) \cap {\downarrow}\{\alpha \to \beta\}$, then $\alpha
\vdash \chi \vdash \alpha \to \beta$, contradicting $\alpha \nvdash \alpha
\to \beta$.
\end{proof}

\subsection{Existence lemmas}

\begin{lemma}[Conditional Existence]\label{lem:cond-exist}
\textup{(Uses the basic rules, \rn{$\wedge{\vdash}\to$}, \rn{Imp},
\rn{Rep}, \rn{DT$_0$}, \rn{Con$\bot$}.)}
\begin{enumerate}[label=(\arabic*)]
\item Let $\Gamma$ be closed with $\alpha \to \beta \notin \Gamma$. Then
$y := (\Th(\alpha), {\downarrow}\{\alpha \to \beta\}) \in W_b$, and it
satisfies $\Gamma \cap \Delta_y = \varnothing$, $\alpha \in \Gamma_y$, and
$\alpha \wedge \beta \in \Delta_y$.
\item Let $x \in W_d$ with $\alpha \to \beta \in \Gamma_x$, and let $y \in
W_d$ with $y \open x$ and $\alpha \in \Gamma_y$. Then $z :=
(\Th(\alpha\wedge\beta), {\downarrow}\{\neg(\alpha \wedge \beta)\}) \in
W_b$, and it satisfies $y \open z$ \textup{(i.e., $\Gamma_z \cap \Delta_y =
\varnothing$)} and $\alpha, \beta \in \Gamma_z$.
\end{enumerate}
\end{lemma}

\begin{proof}
(1) First, $\alpha \nvdash \alpha \to \beta$: otherwise \rn{DT$_0$} gives
$\vdash \alpha \to (\alpha \to \beta)$, and \rn{Imp} (with outer antecedent
$\alpha$; side condition $\alpha \wedge \beta \vdash \alpha$) together with
\rn{Rep} gives $\alpha \to (\alpha \to \beta) \vdash \alpha \wedge \alpha
\to \beta \dashv\vdash \alpha \to \beta$, so $\vdash \alpha \to \beta$ by
\rn{Cut}, whence $\alpha \to \beta \in \Th(\Gamma) \subseteq \Gamma$ by
\rn{Mon}, a contradiction. So $y \in W_b$ by definition. If $\chi \in
\Gamma \cap {\downarrow}\{\alpha \to \beta\}$, then $\Gamma \vdash \chi
\vdash \alpha \to \beta$, so $\alpha \to \beta \in \Gamma$, a
contradiction; so $\Gamma \cap \Delta_y = \varnothing$. Clearly $\alpha \in
\Th(\alpha)$, and $\alpha \wedge \beta \in {\downarrow}\{\alpha \to
\beta\}$ by \rn{$\wedge{\vdash}\to$}.

(2) Since $y \open x$, we have $\Gamma_x \cap \Delta_y = \varnothing$, so
$\alpha \to \beta \notin \Delta_y$. We claim there is no $\chi \in
\Delta_y$ with $\alpha \wedge \beta \vdash \chi$: otherwise $\alpha \wedge
\beta \in \Delta_y$ by predecessor closure, whence $\alpha \to \beta \in
\Delta_y$ by $\to$-closure relative to $\Gamma_y$ (as $\alpha \in
\Gamma_y$), a contradiction. In particular, since $\bot \in \Delta_y$,
$\alpha \wedge \beta \nvdash \bot$, so $\alpha \wedge \beta \nvdash
\neg(\alpha \wedge \beta)$ by \rn{Con$\bot$}, and thus $z \in W_b$.
Moreover, $\Gamma_z \cap \Delta_y = \varnothing$: any $\chi$ in the
intersection would satisfy $\alpha \wedge \beta \vdash \chi \in \Delta_y$,
contradicting the claim. Finally $\alpha, \beta \in \Th(\alpha \wedge
\beta)$ by \rn{$\wedge$E}.
\end{proof}

\begin{lemma}[Fixpoint Existence]\label{lem:fix-exist}
\textup{(Uses the basic rules, \rn{$\vdash\top\to$}, \rn{$\top\to$E},
\rn{DT$_0$}.)} Let $\Gamma$ be closed and $\gamma \notin \Gamma$. Then $y
:= (\Th(\top), {\downarrow}\{\top \to \gamma\}) \in W_b$, and it satisfies
$\gamma \in \Delta_y$ and $\Gamma \cap \Delta_y = \varnothing$.
\end{lemma}

\begin{proof}
Since $\Gamma$ is closed and $\gamma \notin \Gamma$, we have $\nvdash
\gamma$. If $\top \vdash \top \to \gamma$, then $\vdash \top \to \gamma$ by
\rn{$\top$} and \rn{Cut}, so $\vdash \gamma$ by \rn{$\top\to$E} and
\rn{Cut}, a contradiction; hence $y \in W_b$. We have $\gamma \vdash \top
\to \gamma$ by \rn{$\vdash\top\to$}, so $\gamma \in \Delta_y$. If $\chi \in
\Gamma \cap \Delta_y$, then $\Gamma \vdash \chi \vdash \top \to \gamma
\vdash \gamma$, so $\gamma \in \Gamma$, a contradiction.
\end{proof}

\begin{lemma}[Pseudo-Symmetry Existence]\label{lem:ps-exist}
\textup{(Uses the basic rules, \rn{$\neg$Ant}, \rn{$\neg\neg$I}; hence
available for $\vdash\, = \,\vF$.)} Let $x, y \in W_d$ with $y \open x$.
Then $z := (\Gamma_x, {\downarrow}\neg[\Gamma_x]) \in W_a$ satisfies $z
\open y$ and $\Gamma_z = \Gamma_x$; in particular, $z \preref x$ in any
canonical frame $(W, \open)$ with $z \in W$, since $u \open z
\Leftrightarrow u \open x$ for all $u$.
\end{lemma}

\begin{proof}
By Lemma~\ref{lem:least-ct}(1), $z \in W_a \subseteq W_d$. For $z \open y$
we must show $\Gamma_y \cap {\downarrow}\neg[\Gamma_x] = \varnothing$.
Suppose toward a contradiction that $\varphi \in \Gamma_y$ and $\varphi
\vdash \neg\gamma$ for some $\gamma \in \Gamma_x$. By \rn{$\neg$Ant},
$\neg\neg\gamma \vdash \neg\varphi$, and by \rn{$\neg\neg$I} and \rn{Cut},
$\gamma \vdash \neg\varphi$. Since $\gamma \in \Gamma_x$ and $\Gamma_x$ is
closed, $\neg\varphi \in \Gamma_x$. But $\neg\varphi \in
{\downarrow}\neg[\Gamma_y] \subseteq \Delta_y$ by
Lemma~\ref{lem:least-ct}(2) applied to $y$ (as $\varphi \in \Gamma_y$), so
$\neg\varphi \in \Gamma_x \cap \Delta_y$, contradicting $y \open x$.
Finally, since $\Gamma_z = \Gamma_x$, the condition $\Gamma_z \cap \Delta_u
= \varnothing$ defining $u \open z$ is the same as the condition defining
$u \open x$; so $z \preref x$.
\end{proof}

\subsection{Truth sets and the truth lemma}

\begin{lemma}[Truth sets]\label{lem:truth-sets}
Let $W_b \subseteq W \subseteq W_d$ and consider the canonical frame $(W,
\open)$. For all $\alpha, \beta \in \Fm$:
\begin{enumerate}[label=(\arabic*)]
\item $|\top|_W = W$;
\item $|\bot|_W = \varnothing$;
\item $|\alpha \wedge \beta|_W = |\alpha|_W \cap |\beta|_W$;
\item $|\alpha \to \beta|_W = |\alpha|_W \to_{\open} |\beta|_W$;
\item $|\alpha|_W \cup |\beta|_W \subseteq |\alpha \vee \beta|_W \subseteq
c_{\open}(|\alpha|_W \cup |\beta|_W)$.
\end{enumerate}
\end{lemma}

\begin{proof}
We omit the subscript $W$. (1) Since $\vdash\top$, closure of $\Gamma_x$
gives $\top\in\Gamma_x$ for every $x\in W$.

(2) If $\bot \in \Gamma_x$, then $\Gamma_x$ is inconsistent (or intersects
$\Delta_x \ni \bot$), contradicting $x \in W_d$.

(3) By \rn{$\wedge$I}, \rn{$\wedge$E}, and closure of $\Gamma_x$.

(4) ($\subseteq$) Suppose $\alpha \to \beta \in \Gamma_x$ and let $y \open
x$ with $y \in |\alpha|$, i.e., $\alpha \in \Gamma_y$. By
Lemma~\ref{lem:cond-exist}(2), there is $z \in W_b \subseteq W$ with $y
\open z$ and $\alpha, \beta \in \Gamma_z$, i.e., $z \in |\alpha| \cap
|\beta|$. Hence $x \in |\alpha| \to_{\open} |\beta|$.

($\supseteq$) Suppose $\alpha \to \beta \notin \Gamma_x$. By
Lemma~\ref{lem:cond-exist}(1), there is $y \in W_b \subseteq W$ with
$\Gamma_x \cap \Delta_y = \varnothing$ (i.e., $y \open x$), $\alpha \in
\Gamma_y$, and $\alpha \wedge \beta \in \Delta_y$. For any $z$ with $y
\open z$, we have $\Gamma_z \cap \Delta_y = \varnothing$; if $z \in
|\alpha| \cap |\beta|$, then $\alpha \wedge \beta \in \Gamma_z$ by
\rn{$\wedge$I} and closure, so $\alpha \wedge \beta \in \Gamma_z \cap
\Delta_y$, a contradiction. Hence no $z$ with $y \open z$ lies in $|\alpha|
\cap |\beta|$, while $y \in |\alpha|$; so $x \notin |\alpha| \to_{\open}
|\beta|$.

(5) The first inclusion holds by \rn{$\vee$I} and closure. For the second,
suppose $\alpha \vee \beta \in \Gamma_x$ and let $y \open x$; we must find
$z$ with $y \open z$ and $z \in |\alpha| \cup |\beta|$. Since $\alpha \vee
\beta \in \Gamma_x$ and $\Gamma_x \cap \Delta_y = \varnothing$, we have
$\alpha \vee \beta \notin \Delta_y$. We claim $\alpha \notin \Delta_y$ or
$\beta \notin \Delta_y$: otherwise $\alpha \vee \beta \in \Delta_y$ by
$\vee$-closure. Say $\alpha \notin \Delta_y$ (the other case is
symmetric); then by predecessor closure, $\Th(\alpha) \cap \Delta_y =
\varnothing$, and in particular $\bot \in \Delta_y$ gives $\alpha \nvdash
\bot$, so $\alpha \nvdash \neg\alpha$ by \rn{Con$\bot$}. Hence $z :=
(\Th(\alpha), {\downarrow}\{\neg\alpha\}) \in W_b \subseteq W$, and
$\Gamma_z \cap \Delta_y = \Th(\alpha) \cap \Delta_y = \varnothing$, i.e.,
$y \open z$, with $z \in |\alpha|$.
\end{proof}

\begin{lemma}[Fixpoint Lemma]\label{lem:fixpoint}
Let $W_b \subseteq W \subseteq W_d$. Then $|\alpha|_W \in \chi_{\open}(W)$
for every $\alpha \in \Fm$.
\end{lemma}

\begin{proof}
We must show $c_{\open}(|\alpha|) \subseteq |\alpha|$, i.e., if $\alpha
\notin \Gamma_x$ then there is $y \open x$ such that no $z$ with $y \open
z$ satisfies $\alpha \in \Gamma_z$. By Lemma~\ref{lem:fix-exist} applied to
$\Gamma_x$, there is $y \in W_b \subseteq W$ with $\alpha \in \Delta_y$ and
$\Gamma_x \cap \Delta_y = \varnothing$, i.e., $y \open x$. For any $z$ with
$y \open z$: $\Gamma_z \cap \Delta_y = \varnothing$ and $\alpha \in
\Delta_y$ give $\alpha \notin \Gamma_z$.
\end{proof}

\begin{theorem}[Canonical model]\label{thm:canonical}
Let $\vdash\, \in \{\vT, \vF\}$, let $W_c$ be any set with $W_a \cup W_b
\subseteq W_c \subseteq W_d$, and let $\mathcal{M}_c = (W_c, \open, V_c)$
with $V_c(p) = |p|_{W_c}$. Then:
\begin{enumerate}[label=(\arabic*)]
\item $\mathcal{M}_c$ is a model \textup{(}$V_c$ is admissible\textup{)};
\item \textup{(Truth Lemma)} $\true{\alpha}_{\mathcal{M}_c} =
|\alpha|_{W_c}$ for all $\alpha \in \Fm$;
\item $\open$ is reflexive on $W_c$; and if $\vdash\, = \,\vF$, then $(W_c,
\open)$ is an $\SysF$-frame.
\end{enumerate}
\end{theorem}

\begin{proof}
(1) is immediate from Lemma~\ref{lem:fixpoint}.

(2) By induction on $\alpha$, using Lemma~\ref{lem:truth-sets}. The atomic
case holds by definition of $V_c$; the cases $\top$, $\bot$, $\wedge$, $\to$ are
clauses (1)--(4) of Lemma~\ref{lem:truth-sets} together with $\true{\bot} =
\Box\varnothing = \varnothing$ (which holds since $\open$ is reflexive by
part (3), proved independently below; alternatively, $\Box\varnothing
\subseteq \Box\blacklozenge|\bot| = |\bot|$ directly by
Lemma~\ref{lem:fixpoint} and monotonicity of $\Box$). For $\vee$: by
Lemma~\ref{lem:truth-sets}(5), monotonicity of $c_{\open}$, and
Lemma~\ref{lem:fixpoint},
\[
c_{\open}(|\alpha| \cup |\beta|) \subseteq c_{\open}(|\alpha \vee \beta|) =
|\alpha \vee \beta| \subseteq c_{\open}(|\alpha| \cup |\beta|),
\]
so $|\alpha \vee \beta| = c_{\open}(|\alpha| \cup |\beta|) =
c_{\open}(\true{\alpha} \cup \true{\beta}) = \true{\alpha \vee \beta}$ by
the inductive hypothesis.

(3) Reflexivity: for $x \in W_d$, $\Gamma_x \cap \Delta_x = \varnothing$ by
definition of $W_d$, so $x \open x$. Pseudo-symmetry for $\vdash\, =
\,\vF$: given $y \open x$ in $W_c$, Lemma~\ref{lem:ps-exist} provides $z
\in W_a \subseteq W_c$ with $z \open y$ and $z \preref x$.
\end{proof}

\begin{theorem}[Strong completeness]\label{thm:completeness}
For all $\Gamma \subseteq \Fm$ and $\varphi \in \Fm$:
\begin{enumerate}[label=(\arabic*)]
\item $\Gamma \vT \varphi$ \ iff \ $\Gamma \vDash_{\mathsf{C}_{\SysT}}
\varphi$;
\item $\Gamma \vF \varphi$ \ iff \ $\Gamma \vDash_{\mathsf{C}_{\SysF}}
\varphi$.
\end{enumerate}
\end{theorem}

\begin{proof}
Left to right is Theorem~\ref{thm:soundness}. For the converse, suppose
$\Gamma \nvdash \varphi$, where $\vdash\, \in \{\vT, \vF\}$. Let $\Gamma' =
\Th(\Gamma)$; then $\Gamma'$ is closed, consistent (else $\Gamma \vdash
\bot \vdash \varphi$), and $\varphi \notin \Gamma'$. Let $x_0 = (\Gamma',
{\downarrow}\neg[\Gamma']) \in W_a$. Take $W_c = W_a \cup W_b$ and
$\mathcal{M}_c$ as in Theorem~\ref{thm:canonical}. By the Truth Lemma,
$\mathcal{M}_c, x_0 \vDash \gamma$ for each $\gamma \in \Gamma$ (as $\Gamma
\subseteq \Gamma'$) and $\mathcal{M}_c, x_0 \nvDash \varphi$. By
Theorem~\ref{thm:canonical}(3), $\mathcal{M}_c \in \mathsf{C}_{\SysT}$
(resp.~$\mathsf{C}_{\SysF}$ when $\vdash\, = \,\vF$). Hence $\Gamma
\nvDash_{\mathsf{C}_{\SysT}} \varphi$ (resp.~$\Gamma
\nvDash_{\mathsf{C}_{\SysF}} \varphi$).
\end{proof}

\begin{remark}[Algebraic representation]\label{rem:representation}
Combining Theorem~\ref{thm:completeness} with Theorem~\ref{thm:algebraic}
shows that the equational consequences of the class of $\SysF$-algebras and
of the class of dual algebras of $\SysF$-frames coincide. This belongs to
the tradition of bounded-lattice representation developed in
\cite{Urquhart1978,Ploscica1995}. More specifically, a direct algebra-level
representation theorem---every $\SysF$-algebra embeds into the dual algebra
of an $\SysF$-frame---can be extracted from the representation of arbitrary
bounded lattices with preconditionals in
\cite[\S~3]{Holliday2025} by restricting to \emph{disjoint} filter-ideal
pairs, in analogy with \cite[Thm.~2.15]{HM2026}; since we do not need this
finer result, we leave the details to the interested reader.
\end{remark}

\section{The finite model property}\label{sec:fmp}

In this section we refine the construction of Section~\ref{sec:canonical}
to obtain, for each pair of formulas $\varphi, \psi$ with $\varphi
\nvdash \psi$, a \emph{finite} countermodel. Throughout, $\vdash$ again
denotes $\vT$ or $\vF$. Fix formulas $\varphi_0, \psi_0 \in \Fm$; we aim to
decide $\varphi_0 \vdash \psi_0$.

\subsection{The closure set \texorpdfstring{$\Sigma$}{Sigma}}

\begin{definition}\label{def:Sigma}
Let $\Sigma$ be the smallest set of formulas such that:
\begin{enumerate}[label=(\roman*)]
\item $\varphi_0, \psi_0, \bot, \top \in \Sigma$;
\item $\Sigma$ is closed under subformulas;
\item if $\alpha \to \beta \in \Sigma$, then $\alpha \wedge \beta \in
\Sigma$.
\end{enumerate}
\end{definition}

\begin{lemma}\label{lem:Sigma-finite}
$\Sigma$ is finite; indeed $|\Sigma| \leq 2\,|{\SF(\varphi_0) \cup
\SF(\psi_0)}| + 4$, where $\SF$ denotes the set of subformulas.
\end{lemma}

\begin{proof}
Let $\Sigma_0 = \SF(\varphi_0) \cup \SF(\psi_0) \cup \{\bot, \top\}$.
We claim $\Sigma \subseteq \Sigma_0 \cup \{\alpha
\wedge \beta \mid \alpha \to \beta \in \Sigma_0\} =: \Sigma_1$. It suffices
to check that $\Sigma_1$ satisfies (i)--(iii). For (ii): the subformulas of
a new conjunction $\alpha \wedge \beta$ (with $\alpha \to \beta \in
\Sigma_0$) are $\alpha \wedge \beta$ itself together with subformulas of
$\alpha$ and $\beta$, all of which lie in $\Sigma_0$ by subformula closure
of $\Sigma_0$. For (iii): every conditional in $\Sigma_1$ already lies in
$\Sigma_0$, since the added formulas are conjunctions and contribute no new
conditionals as members of $\Sigma_1$; so its companion conjunction is in
$\Sigma_1$ by construction.
\end{proof}

\subsection{The finite canonical model}

\begin{definition}\label{def:Sigma-notions}
For $\Gamma \subseteq \Fm$ and $\Xi \subseteq \Fm$, define
\[
\Th_{\Sigma}(\Gamma) = \{\chi \in \Sigma \mid \Gamma \vdash \chi\},
\qquad
{\downarrow_{\Sigma}}\Xi = \{\chi \in \Sigma \mid \chi \vdash \psi \text{
for some } \psi \in \Xi\}.
\]
Say $\Gamma$ is \emph{$\Sigma$-closed} if $\Th_{\Sigma}(\Gamma) \subseteq \Gamma$.
Say $\Delta \subseteq \Sigma$ is \emph{$\to_{\Sigma}$-closed
relative to $\Gamma$} if for all $\alpha \in \Gamma$ and $\beta \in \Fm$
with $\alpha \to \beta \in \Sigma$: if $\alpha \wedge \beta \vdash \chi$
for some $\chi \in \Delta$, then $\alpha \to \beta \in \Delta$.
\begin{align*}
W^{\Sigma}_d &= \bigl\{(\Gamma, \Delta) \;\big|\; \Gamma, \Delta \subseteq
\Sigma;\ \Gamma \text{ is $\Sigma$-closed};\ \bot \in \Delta;\ \Delta
\text{ is closed under $\Sigma$-predecessors}\\
&\qquad\quad \text{(${\downarrow_{\Sigma}}\Delta \subseteq \Delta$) and
under $\Sigma$-joins (if $\chi_1, \chi_2 \in \Delta$ and $\chi_1 \vee
\chi_2 \in \Sigma$, then $\chi_1 \vee \chi_2 \in \Delta$);}\\
&\qquad\quad \Delta \text{ is $\to_{\Sigma}$-closed relative to } \Gamma;
\text{ and } \Gamma \cap \Delta = \varnothing \bigr\};\\
W^{\Sigma}_b &= \bigl\{(\Th_{\Sigma}(\alpha),
{\downarrow_{\Sigma}}\{\alpha \to \beta\}) \;\big|\; \alpha, \beta \in
\Sigma,\ \alpha \nvdash \alpha \to \beta \bigr\},
\end{align*}
with the openness relation $y \open x$ iff $\Gamma_x \cap \Delta_y =
\varnothing$, and $|\chi|_W = \{x \in W \mid \chi \in \Gamma_x\}$ as
before. Note that $W^{\Sigma}_d$ is finite, of size at most $4^{|\Sigma|}$.
\end{definition}

Note that this is the derivability-based formulation used in the previous
chapter, now relativized to $\Sigma$. The companion-closure condition on
$\Sigma$ (if $\alpha \to \beta \in \Sigma$, then $\alpha \wedge \beta \in \Sigma$)
ensures that the relevant conjunctions stay inside $\Sigma$.
\begin{lemma}\label{lem:Wb-in-Wd-fin}
\textup{(Uses the basic rules, \rn{CAS$'$}.)} $W^{\Sigma}_b \subseteq
W^{\Sigma}_d$.
\end{lemma}

\begin{proof}
This is Lemma~\ref{lem:Wb-in-Wd} restricted to $\Sigma$.  The same
argument applies to $\Th_{\Sigma}$ and $\downarrow_{\Sigma}$ throughout:
every formula quantified over in the closure clauses is now required to
belong to $\Sigma$, and the conclusions of the corresponding steps in the
proof of Lemma~\ref{lem:Wb-in-Wd} are precisely such formulas.
\end{proof}

\begin{lemma}[$\Sigma$-Conditional Existence]\label{lem:cond-exist-fin}
\textup{(Rules as in Lemma~\ref{lem:cond-exist}.)}
\begin{enumerate}[label=(\arabic*)]
\item Let $\alpha \to \beta \in \Sigma$ and let $\Gamma$ be
$\Sigma$-closed with $\alpha \to \beta \notin \Gamma$. Then $y :=
(\Th_{\Sigma}(\alpha), {\downarrow_{\Sigma}}\{\alpha \to \beta\}) \in
W^{\Sigma}_b$, with $\Gamma \cap \Delta_y = \varnothing$, $\alpha \in
\Gamma_y$, and $\alpha \wedge \beta \in \Delta_y$.
\item Let $x \in W^{\Sigma}_d$ with $\alpha \to \beta \in \Gamma_x$, and
let $y \in W^{\Sigma}_d$ with $y \open x$ and $\alpha \in \Gamma_y$. Then
$z := (\Th_{\Sigma}(\alpha \wedge \beta),
{\downarrow_{\Sigma}}\{\neg(\alpha \wedge \beta)\}) \in W^{\Sigma}_b$, with
$y \open z$ and $\alpha, \beta \in \Gamma_z$.
\end{enumerate}
\end{lemma}

\begin{proof}
Both parts are the restrictions to $\Sigma$ of the corresponding parts of
Lemma~\ref{lem:cond-exist}.  In (1), subformula closure gives $\alpha,\beta
\in\Sigma$, companion closure gives $\alpha\wedge\beta\in\Sigma$, and
$\Sigma$-closure of $\Gamma$ replaces closure of $\Gamma$ in the
theorem-based contradiction.  In (2), $\alpha\to\beta\in\Gamma_x\subseteq
\Sigma$ again gives $\alpha\wedge\beta\in\Sigma$; the original argument
then uses only $\Sigma$-predecessor and $\to_{\Sigma}$-closure.  Thus its
witnesses are the displayed members of $W^{\Sigma}_b$.
\end{proof}

\begin{lemma}[$\Sigma$-Fixpoint Existence]\label{lem:fix-exist-fin}
\textup{(Rules as in Lemma~\ref{lem:fix-exist}.)} Let $\gamma \in \Sigma$
and let $\Gamma$ be $\Sigma$-closed with $\gamma \notin \Gamma$. Then $y :=
(\Th_{\Sigma}(\top), {\downarrow_{\Sigma}}\{\top \to \gamma\}) \in
W^{\Sigma}_b$, with $\gamma \in \Delta_y$ and $\Gamma \cap \Delta_y =
\varnothing$.
\end{lemma}

\begin{proof}
This is Lemma~\ref{lem:fix-exist} restricted to $\Sigma$.  Here
$\top,\gamma\in\Sigma$, and the only use of closure of $\Gamma$ in that
proof is to infer membership from the theorem $\vdash\gamma$, which is
equally supplied by $\Sigma$-closure.
\end{proof}

\begin{lemma}[$\Sigma$-Pseudo-Symmetry Existence]\label{lem:ps-exist-fin}
\textup{(Uses the basic rules, \rn{$\to$Mon}, \rn{$\neg$Ant},
\rn{$\neg\neg$I}, \rn{Con$\bot$}; hence available for $\vdash\, =
\,\vF$.)} Let $x, y \in W^{\Sigma}_b$ with $y \open x$, say $x =
(\Th_{\Sigma}(\alpha), {\downarrow_{\Sigma}}\{\alpha \to \beta\})$ and $y =
(\Th_{\Sigma}(\varphi), {\downarrow_{\Sigma}}\{\varphi \to \psi\})$. Then
$z := (\Th_{\Sigma}(\alpha), {\downarrow_{\Sigma}}\{\neg\alpha\}) \in
W^{\Sigma}_b$ satisfies $z \open y$ and $\Gamma_z = \Gamma_x$; in
particular $z \preref x$ in $(W^{\Sigma}_b, \open)$.
\end{lemma}

\begin{proof}
Since $\alpha \nvdash \alpha \to \beta$, \rn{$\bot$E} and \rn{Cut} give
$\alpha \nvdash \bot$, so $\alpha \nvdash \neg\alpha$ by \rn{Con$\bot$};
with $\alpha, \bot \in \Sigma$, we get $z \in W^{\Sigma}_b$. For $z \open
y$, suppose toward a contradiction that $\chi \in \Th_{\Sigma}(\varphi)
\cap {\downarrow_{\Sigma}}\{\neg\alpha\}$. Then $\varphi \vdash \chi \vdash
\neg\alpha$, so by \rn{$\neg$Ant} and \rn{$\neg\neg$I} with \rn{Cut},
$\alpha \vdash \neg\neg\alpha \vdash \neg\varphi$. Since $\neg\varphi =
\varphi \to \bot \vdash \varphi \to \psi$ by \rn{$\to$Mon} and
\rn{$\bot$E}, we get $\alpha \vdash \varphi \to \psi$. But then $\alpha \in
\Gamma_x \cap \Delta_y$ (as $\alpha \in \Sigma$), contradicting $y \open
x$. Finally $\Gamma_z = \Th_{\Sigma}(\alpha) = \Gamma_x$, so $u \open z
\Leftrightarrow u \open x$ for all $u$, whence $z \preref x$.
\end{proof}

\subsection{Truth lemma and the finite model property}

\begin{lemma}[$\Sigma$-Truth sets]\label{lem:truth-sets-fin}
Let $W := W^{\Sigma}_b$ with the canonical frame $(W, \open)$. For all
formulas:
\begin{enumerate}[label=(\arabic*)]
\item $|\bot|_W = \varnothing$;
\item if $\alpha \wedge \beta \in \Sigma$, then $|\alpha \wedge \beta|_W =
|\alpha|_W \cap |\beta|_W$;
\item if $\alpha \to \beta \in \Sigma$, then $|\alpha \to \beta|_W =
|\alpha|_W \to_{\open} |\beta|_W$;
\item if $\alpha \vee \beta \in \Sigma$, then $|\alpha|_W \cup |\beta|_W
\subseteq |\alpha \vee \beta|_W \subseteq c_{\open}(|\alpha|_W \cup
|\beta|_W)$;
\item $|\gamma|_W \in \chi_{\open}(W)$ for every $\gamma \in \Sigma$.
\end{enumerate}
\end{lemma}

\begin{proof}
Items (1)--(4) are the restrictions to $W^{\Sigma}_b$ of the corresponding
items of Lemma~\ref{lem:truth-sets}.  The required witnesses are supplied
by Lemma~\ref{lem:cond-exist-fin}; all applications of theory, predecessor,
and join closure are legitimate in their $\Sigma$-restricted forms, since
the hypotheses of (2)--(4), together with subformula and companion closure,
place the relevant formulas in $\Sigma$.  Item (5) is the same proof as
Lemma~\ref{lem:fixpoint}, using Lemma~\ref{lem:fix-exist-fin}.
\end{proof}

\begin{theorem}[Finite canonical model]\label{thm:fmp-canonical}
Let $\vdash\, \in \{\vT, \vF\}$ and let $\mathcal{M}^{\Sigma} =
(W^{\Sigma}_b, \open, V)$ with $V(p) = |p|_{W^{\Sigma}_b}$ for $p \in
\Sigma$ and $V(p) = \varnothing$ otherwise. Then $\mathcal{M}^{\Sigma}$ is
a finite model; $\true{\gamma}_{\mathcal{M}^{\Sigma}} =
|\gamma|_{W^{\Sigma}_b}$ for all $\gamma \in \Sigma$; the frame is
reflexive; and if $\vdash\, = \,\vF$, it is an $\SysF$-frame.
\end{theorem}

\begin{proof}
Admissibility of $V$: for $p \in \Sigma$ by
Lemma~\ref{lem:truth-sets-fin}(5); $\varnothing$ is a proposition on a
reflexive frame, since there are no absurd states. Reflexivity: for $x \in
W^{\Sigma}_b$, $\Gamma_x \cap \Delta_x = \varnothing$
(Lemma~\ref{lem:Wb-in-Wd-fin}), so $x \open x$. The truth lemma for $\gamma
\in \Sigma$ is by induction on $\gamma$ exactly as in
Theorem~\ref{thm:canonical}(2), using Lemma~\ref{lem:truth-sets-fin} and
subformula closure of $\Sigma$. Pseudo-symmetry for $\vdash\, = \,\vF$ is
Lemma~\ref{lem:ps-exist-fin}, whose witness lies in $W^{\Sigma}_b$.
\end{proof}

\begin{theorem}[Finite model property]\label{thm:fmp}
Let $\SysS \in \{\SysT, \SysF\}$ and $\varphi_0, \psi_0 \in \Fm$. Then
$\varphi_0 \vdash_{\SysS} \psi_0$ iff $\varphi_0 \vDash \psi_0$ over all
\emph{finite} models in $\mathsf{C}_{\SysS}$. Moreover, if $\varphi_0
\nvdash_{\SysS} \psi_0$, there is a countermodel with at most
$4^{|\Sigma|}$ states, where $|\Sigma| \leq 2\,|{\SF(\varphi_0) \cup
\SF(\psi_0)}| + 4$.
\end{theorem}

\begin{proof}
Left to right is Theorem~\ref{thm:soundness}. Conversely, suppose
$\varphi_0 \nvdash \psi_0$. Then $\varphi_0 \nvdash \bot$ (else \rn{$\bot$E}
and \rn{Cut} give $\varphi_0 \vdash \psi_0$), so $\varphi_0 \nvdash
\neg\varphi_0$ by \rn{Con$\bot$}, and $x_0 := (\Th_{\Sigma}(\varphi_0),
{\downarrow_{\Sigma}}\{\neg\varphi_0\}) \in W^{\Sigma}_b$. We have
$\varphi_0 \in \Gamma_{x_0}$ and $\psi_0 \notin \Gamma_{x_0}$ (as $\psi_0
\in \Sigma$ and $\varphi_0 \nvdash \psi_0$). By
Theorem~\ref{thm:fmp-canonical}, $\mathcal{M}^{\Sigma}, x_0 \vDash
\varphi_0$ and $\mathcal{M}^{\Sigma}, x_0 \nvDash \psi_0$, and
$\mathcal{M}^{\Sigma} \in \mathsf{C}_{\SysS}$ is finite of the stated size.
\end{proof}

\begin{corollary}[Decidability]\label{cor:decidable}
For $\SysS \in \{\SysT, \SysF\}$, the relation $\{(\varphi, \psi) \mid
\varphi \vdash_{\SysS} \psi\}$ is decidable.
\end{corollary}

\begin{proof}
The relation is recursively enumerable, since proofs are finite objects
(Remark~\ref{rem:mon-fin}). Its complement is also recursively enumerable:
by Theorem~\ref{thm:fmp} it suffices to search for a finite countermodel,
and given a finite structure $(X, \open, V)$ one can effectively check
reflexivity, pseudo-symmetry, admissibility of $V$, and the truth of
$\varphi$ and falsity of $\psi$ at some state.
\end{proof}

\begin{remark}
For fundamental logic itself, Aguilera and Byd\v{z}ovsk\'{y} \cite{AB2023}
established decidability in polynomial time via cut-free sequent calculi.
Whether $\vT$ or $\vF$ admits an analogous calculus and complexity bound is
an interesting open question; the bound in Theorem~\ref{thm:fmp} yields
only a (weak) exponential-space upper bound.
\end{remark}


\section{Interlude: factoring conditions}\label{sec:strongps}

In the remaining two sections we extend the two main translation theorems
of \cite{HM2026} to the language with the preconditional: we prove that
the GMT-style adaptation whose preconditional clause boxes the
\emph{Sasaki hook} fully and faithfully embeds $\sysF$ into ortho-S4
(Section \ref{sec:os4}), and that the Goldblatt-style translation whose
preconditional clause is the intuitionistic-modal continuation of the
Goldblatt image of that hook fully and faithfully embeds $\sysF$ into
intuitionistic KTB (Section \ref{sec:itb}).

The division of labour is as follows.
\begin{enumerate}
\item In each case, the transformation from target models to fundamental
models is, at the level of frames, exactly the \emph{fundamental reduct}
of \cite[\S\S~3.2, 4.2]{HM2026}: the openness relation is the composite
$\lhd \,=\, \rho\circ{\le}$ of the symmetric relation $\rho$ of the target
frame with its preorder $\le$.  The additional calculation required for the
present language is the transfer of the semantic clause of the preconditional,
\[
\to_{\lhd}(A,B)\;=\;\Box_{\lhd}\bigl(-A\cup\Diamond_{\rhd}(A\cap B)\bigr),
\]
which factors through the two decompositions of $\lhd$ to yield precisely
the \emph{boxed Sasaki hook} on the ortho-side and the strict clause
representing its Goldblatt image on the intuitionistic side
(Theorems~\ref{thm:F1} and~\ref{thm:F2}).
\item In the converse direction, instead of the general modal-companion
constructions of \cite[\S\S~3.3, 4.3]{HM2026}, we use same-carrier
constructions $G_1, G_2$ that are available when the fundamental frame
satisfies certain \emph{factoring} conditions
(Definition \ref{def:strong-ps}).  Under these conditions the identities
$F_1(G_1(\mathcal N))=\mathcal N$ and $F_2(G_2(\mathcal N))=\mathcal N$
hold \emph{on the nose}, so the semantic clause for the preconditional is
inherited from the forward direction at no extra cost
(Theorems~\ref{thm:G1} and~\ref{thm:G2}).
\item The canonical model of Section \ref{sec:canonical}, taken over the
carrier $W_a\cup W_b\cup W_e$, satisfies the required factoring conditions (Theorem \ref{thm:canonical-strongps}).  This is the
technical heart of the faithfulness proofs: it shows that the canonical
construction driven by the preconditional---rather than by a primitive
negation---still supplies enough symmetry for the round trip through the
modal companions.
\item The main theorems are then assembled from soundness and
completeness on both sides (Theorems~\ref{thm:main-os4}
and~\ref{thm:main-itb}).
\end{enumerate}

The present section collects the frame-level preliminaries: a small
toolbox of operator identities (\S\ref{subsec:toolbox}), the factoring conditions and their relation-algebraic characterizations
(\S\ref{subsec:strongps}), and the proof that the canonical model
satisfies them (\S\ref{subsec:canonical-strongps}).

\subsection{Operators on relational frames}\label{subsec:toolbox}

Throughout, $X$ is a nonempty set and $-U:=X\setminus U$ for $U\subseteq X$.
For a binary relation $Q$ on $X$, written infix, we use the
\emph{predecessor-style} modal operators
\[
\Box_{Q}U \;=\;\{x\in X\mid \forall y\,(y\mathrel{Q}x\Rightarrow y\in U)\},
\qquad
\Diamond_{Q}U \;=\;\{x\in X\mid \exists y\,(y\mathrel{Q}x \text{ and } y\in U)\},
\]
we write $Q^{-1}$ for the converse of $Q$, and we compose relations by
\[
y\,(Q\circ R)\,x \iff \exists z\,(y\mathrel{Q}z \text{ and } z\mathrel{R}x).
\]
Recall from Section \ref{sec:relational} that for a relational frame
$(X,\lhd)$ we write $\rhd:={\lhd}^{-1}$,
\[
c_{\lhd}\;=\;\Box_{\lhd}\Diamond_{\rhd},\qquad
\neg_{\lhd}U\;=\;\Box_{\lhd}(-U),\qquad
\to_{\lhd}(A,B)\;=\;\Box_{\lhd}\bigl(-A\cup\Diamond_{\rhd}(A\cap B)\bigr),
\]
that the \emph{propositions} of $(X,\lhd)$ are the fixpoints of the
closure operator $c_{\lhd}$, and that the semantic clauses of
Section~\ref{sec:relational} read
\[
\den{\alpha\wedge\beta}=\den{\alpha}\cap\den{\beta},\qquad
\den{\alpha\vee\beta}=c_{\lhd}(\den{\alpha}\cup\den{\beta}),\qquad
\den{\alpha\to\beta}=\to_{\lhd}(\den{\alpha},\den{\beta}),
\]
with $\den{\bot}$ the set of absurd states (which is empty over reflexive
frames) and $\den{p}=V(p)\in\FP(c_{\lhd})$.
Recall also the \emph{prerefinement} preorder: $y\pre x$ iff every
$z\lhd y$ satisfies $z\lhd x$.  We additionally need its mirror image, the
\emph{postrefinement} preorder:
\[
y\post x \iff \forall z\,(y\lhd z\Rightarrow x\lhd z),
\]
i.e., prerefinement computed in the converse frame $(X,\rhd)$.  Two facts
are immediate from the definitions and will be used without mention:
\[
y\pre x \text{ and } z\lhd y \;\Longrightarrow\; z\lhd x,
\qquad\qquad
y\post x \text{ and } y\lhd z \;\Longrightarrow\; x\lhd z.
\]

\begin{lemma}[Toolbox]\label{lem:toolbox}
Let $Q,Q',R$ be binary relations on $X$ and $U,V\subseteq X$.
\begin{enumerate}
\item[(a)] $\Box_{Q\circ R}=\Box_{R}\Box_{Q}$ and
$\Diamond_{Q\circ R}=\Diamond_{R}\Diamond_{Q}$.
\item[(b)] If $Q\subseteq Q'$, then $\Box_{Q'}U\subseteq\Box_{Q}U$ and
$\Diamond_{Q}U\subseteq\Diamond_{Q'}U$.  Conversely, if
$\Box_{Q'}Z\subseteq\Box_{Q}Z$ for \emph{all} $Z\subseteq X$, then
$Q\subseteq Q'$.
\item[(c)] If $Q$ is reflexive, then $\Box_{Q}U\subseteq U\subseteq\Diamond_{Q}U$.
\item[(d)] If $Q$ is a preorder, then $\Box_{Q}\Box_{Q}=\Box_{Q}$,
$\Diamond_{Q}\Diamond_{Q}=\Diamond_{Q}$, and
$\Diamond_{Q^{-1}}\Box_{Q}=\Box_{Q}$.
\item[(e)] If $Q$ is symmetric, then $U\subseteq\Box_{Q}\Diamond_{Q}U$ and
$\Box_{Q}\Diamond_{Q}\Box_{Q}=\Box_{Q}$.
\item[(f)] $\Box_{Q}(U\cap V)=\Box_{Q}U\cap\Box_{Q}V$ and
$\Diamond_{Q}(U\cup V)=\Diamond_{Q}U\cup\Diamond_{Q}V$; more generally,
$\Box_{Q}$ distributes over arbitrary intersections and $\Diamond_{Q}$
over arbitrary unions.
\item[(g)] For any relational frame $(X,\lhd)$, the operator
$c_{\lhd}=\Box_{\lhd}\Diamond_{\rhd}$ is a closure operator, and
\[
\FP(c_{\lhd})\;=\;\Box_{\lhd}[\wp(X)]\;=\;\{\Box_{\lhd}Z\mid Z\subseteq X\}.
\]
\end{enumerate}
\end{lemma}

\begin{proof}
For (a), (b), and (f), expand the defining quantifiers; the converse in
(b) follows by taking $Z=X\setminus\{y\}$ if $yQx$ but not $yQ'x$.
For (c)--(e), use the reflexive edge $xQx$, transitivity for the two
idempotence claims in (d), and symmetry to reverse the edge used to witness
the two assertions in (e).

For (g), $\Diamond_{\rhd}$ is left adjoint to $\Box_{\lhd}$:
\[
\Diamond_{\rhd}A\subseteq B
\quad\Longleftrightarrow\quad
A\subseteq\Box_{\lhd}B.
\]
Hence $c_{\lhd}=\Box_{\lhd}\Diamond_{\rhd}$ is the closure operator induced
by this Galois connection.  Its fixed points are exactly the elements in the
image of its right adjoint, i.e.,
$\FP(c_{\lhd})=\Box_{\lhd}[\wp(X)]$; see also
\cite[Lemma 2.9, Prop.\ 2.10]{HM2026}.
\end{proof}

The reflexivity and prefactoring argument below is the common relational core
of \cite[Lemmas~3.5 and~4.7]{HM2026}.  We record it in an assumption-minimal
form, together with the operator- and proposition-level calculations needed
for the conditional, so that one lemma applies to both reducts.

\begin{lemma}[Reduct frames]\label{lem:reduct}
Let $\le$ be a preorder on $X$, let $\rho$ be a reflexive symmetric
relation on $X$, and let $\lhd:=\rho\circ{\le}$, i.e.,
\[
y\lhd x \iff \exists u\,(y\mathrel{\rho}u \text{ and } u\le x).
\]
Then:
\begin{enumerate}
\item $\lhd$ is reflexive;
\item $\Box_{\lhd}=\Box_{\le}\Box_{\rho}$ and
$\Diamond_{\rhd}=\Diamond_{\rho}\Diamond_{\le^{-1}}$;
\item $(X,\lhd)$ is prefactoring in the sense of
Definition \ref{def:strong-ps} below; in particular, it is
pseudo-symmetric.
\item
\[
\FP(c_{\lhd})
=\Box_{\le}\Box_{\rho}[\wp(X)]
=\Box_{\le}[\FP(c_{\rho})],
\]
and every $A\in\FP(c_{\lhd})$ satisfies $\Box_{\le}A=A$.
\item If $\Box_{\le}U=U$, then
\[
\Diamond_{\le^{-1}}U=U
\qquad\text{and}\qquad
\Diamond_{\rhd}U=\Diamond_{\rho}U.
\]
\item For all $A,B\in\FP(c_{\lhd})$,
\[
\Diamond_{\rhd}(A\cap B)=\Diamond_{\rho}(A\cap B),\qquad
\Diamond_{\rhd}(A\cup B)=\Diamond_{\rho}(A\cup B),
\]
\[
c_{\lhd}(A\cap B)=\Box_{\le}c_{\rho}(A\cap B),\qquad
c_{\lhd}(A\cup B)=\Box_{\le}c_{\rho}(A\cup B),
\]
and
\[
\to_{\lhd}(A,B)
=\Box_{\le}\Box_{\rho}\bigl(-A\cup\Diamond_{\rho}(A\cap B)\bigr).
\]
\end{enumerate}
\end{lemma}

\begin{proof}
(1) $x\mathrel{\rho}x$ and $x\le x$ give $x\lhd x$.

(2) The first identity is Lemma \ref{lem:toolbox}(a).  For the second,
observe that $\rhd=(\rho\circ{\le})^{-1}={\le^{-1}}\circ\rho^{-1}
={\le^{-1}}\circ\rho$ by symmetry of $\rho$, and apply
Lemma \ref{lem:toolbox}(a) again.

(3) Suppose $y\lhd x$, and pick $u$ with $y\mathrel{\rho}u$ and $u\le x$.
First, $u\lhds y$: indeed $u\lhd y$ since $u\mathrel{\rho}y$ (symmetry)
and $y\le y$, and $y\lhd u$ since $y\mathrel{\rho}u$ and $u\le u$.
Second, $u\pre x$: if $z\lhd u$, pick $v$ with $z\mathrel{\rho}v$ and
$v\le u$; then $v\le x$ by transitivity, so $z\lhd x$.  Thus $u$ is the
required witness.  Pseudo-symmetry follows since $u\lhd y$ and $u$
prerefines $x$.

(4) By Lemma \ref{lem:toolbox}(g) and (2),
\[
\FP(c_{\lhd})=\Box_{\lhd}[\wp(X)]
=\Box_{\le}\Box_{\rho}[\wp(X)]
=\Box_{\le}[\FP(c_{\rho})].
\]
If $A=\Box_{\le}\Box_{\rho}Z$, then
$\Box_{\le}A=A$ by Lemma \ref{lem:toolbox}(d).

(5) If $\Box_{\le}U=U$, Lemma \ref{lem:toolbox}(d) gives
$\Diamond_{\le^{-1}}U=U$; the second identity then follows from (2).

(6) The $\Box_{\le}$-fixpoints are the downsets of the preorder $\le$,
so they are closed under intersections and unions.  Hence (4) and (5)
give the two $\Diamond$ identities.  Using these identities together
with (2) in the definitions of $c_{\lhd}$ and $\to_{\lhd}$ gives the
remaining displayed formulas.
\end{proof}

\subsection{Factoring conditions}\label{subsec:strongps}

Recall that a frame $(X,\lhd)$ is pseudo-symmetric if for all $x$ and
$y\lhd x$ there is some $z\lhd y$ with $z\pre x$.  The following
conditions use the symmetric kernel as in \cite[Def.~3.8]{HM2026} and strengthen pseudo-symmetry by requiring the witness $z$ to be compatible with $y$, and possibly to refine $x$ in both directions at once.  For the latter we introduce
\[
y\prep x \;:\iff\; y\pre x \text{ and } x\post y ,
\]
a preorder (as an intersection of two preorders, cf.\ Lemma
\ref{lem:prep} below); when $y\prep x$ we say that $y$ \emph{strongly
refines} $x$.

\begin{definition}\label{def:strong-ps}
Let $(X,\lhd)$ be a relational frame and $\lhds:={\lhd}\cap{\rhd}$ the
symmetric kernel of $\lhd$.  We say that $(X,\lhd)$ is:
\begin{enumerate}
\item \emph{prefactoring} if for all $x\in X$ and
$y\lhd x$ there is $z$ with $z\lhds y$ and $z\pre x$;
\item \emph{strongly factoring} if for all $x\in X$ and
$y\lhd x$ there is $z$ with $z\lhds y$ and $z\prep x$;
\item \emph{postfactoring} if the converse frame $(X,\rhd)$ is prefactoring;
\item \emph{converse-strongly factoring} if the converse frame $(X,\rhd)$ is strongly factoring;
\item \emph{balanced} if it is both prefactoring
and postfactoring.
\end{enumerate}
\end{definition}

Unfolding the definitions (and noting that ${\rhd}^{s}={\lhds}$, that
prerefinement for $(X,\rhd)$ is postrefinement for $(X,\lhd)$, and that
strong refinement for $(X,\rhd)$ is the converse of strong refinement
for $(X,\lhd)$), the converse conditions read: $(X,\lhd)$ is postfactoring iff
\[
\forall x\,\forall y\ \bigl(x\lhd y \;\Rightarrow\; \exists z\,
(z\lhds y \text{ and } z\post x)\bigr),
\]
and converse-strongly factoring iff
\[
\forall x\,\forall y\ \bigl(x\lhd y \;\Rightarrow\; \exists z\,
(z\lhds y \text{ and } z\post x \text{ and } x\pre z)\bigr),
\qquad\text{i.e.,}\quad x\prep z .
\]
Clearly strongly factoring implies prefactoring, and converse-strongly factoring implies postfactoring. Moreover, prefactoring implies pseudo-symmetry, since $z\lhds y$ entails $z\lhd y$.

\begin{lemma}\label{lem:prep}
$\pre$, $\post$, and $\prep$ are preorders, and:
\begin{enumerate}
\item ${\lhds}\circ{\pre}\;\subseteq\;{\lhd}$ and
${\prep}\circ{\lhds}\;\subseteq\;{\lhd}$;
\item ${\lhds}\circ{\prep}\;\subseteq\;{\lhd}$ and
$({\post})^{-1}\circ{\lhds}\;\subseteq\;{\lhd}$.
\end{enumerate}
\end{lemma}

\begin{proof}
Reflexivity and transitivity of $\pre$ and $\post$ are immediate, and
$\prep$ is an intersection-like combination of the two: if $z\prep y$ and
$y\prep x$, then $z\pre y\pre x$ gives $z\pre x$, and $x\post y\post z$
gives $x\post z$, so $z\prep x$.

(1) If $y\lhds z$ and $z\pre x$, then $y\lhd z$ and $z\pre x$, so
$y\lhd x$.  If $y\prep z$ and $z\lhds x$, then in particular
$z\post y$ and $z\lhd x$; hence $y\lhd x$ by the defining property of
postrefinement.

(2) The first inclusion follows from the first inclusion in (1) since
${\prep}\subseteq{\pre}$.  For the second, suppose $z\post y$ and
$z\lhds x$; then $z\lhd x$ and $z\post y$ give $y\lhd x$.
\end{proof}

\begin{proposition}[Relation-algebraic characterizations]\label{prop:strongps-char}
Let $(X,\lhd)$ be a relational frame.
\begin{enumerate}
\item $(X,\lhd)$ is prefactoring iff
${\lhd}={\lhds}\circ{\pre}$; in that case
\[
\Box_{\lhd}\;=\;\Box_{\pre}\,\Box_{\lhds}.
\]
\item $(X,\lhd)$ is strongly factoring iff
${\lhd}={\lhds}\circ{\prep}$; in that case
\[
\Box_{\lhd}\;=\;\Box_{\prep}\,\Box_{\lhds}
\qquad\text{and}\qquad
\Diamond_{\rhd}U\;=\;\Diamond_{\lhds}\bigl(\Diamond_{(\prep)^{-1}}U\bigr)
\ \text{ for all } U\subseteq X.
\]
\item $(X,\lhd)$ is postfactoring iff
${\lhd}=({\post})^{-1}\circ{\lhds}$; in that case
\[
\Box_{\lhd}\;=\;\Box_{\lhds}\,\Box_{({\post})^{-1}},
\qquad\text{where}\quad
\Box_{({\post})^{-1}}U=\{x\mid \forall y\,(x\post y\Rightarrow y\in U)\}.
\]
\end{enumerate}
\end{proposition}

\begin{proof}
(1) The right-to-left inclusion ${\lhds}\circ{\pre}\subseteq{\lhd}$
always holds by Lemma \ref{lem:prep}(1), and the left-to-right inclusion
${\lhd}\subseteq{\lhds}\circ{\pre}$ is a literal restatement of
prefactoring.  The operator identity is Lemma
\ref{lem:toolbox}(a).

(2) Same argument, using ${\lhds}\circ{\prep}\subseteq{\lhd}$ from Lemma
\ref{lem:prep}(2).  For the $\Diamond_{\rhd}$ identity: for any
$x\in X$,
\[
x\in\Diamond_{\rhd}U
\iff \exists y\,(x\lhd y \text{ and } y\in U)
\iff \exists z\,\exists y\,(x\lhds z,\ z\prep y,\ y\in U),
\]
using ${\lhd}={\lhds}\circ{\prep}$; since $\lhds$ is symmetric, the
right-hand side says exactly that some $z\lhds x$ lies in
$\Diamond_{(\prep)^{-1}}U=\{z\mid\exists y\,(z\prep y\text{ and }y\in U)\}$.

(3) Right-to-left is Lemma \ref{lem:prep}(2).  Left-to-right: the
condition displayed after Definition \ref{def:strong-ps} says that
whenever $x\lhd y$ there is $z$ with $z\post x$ and $z\lhds y$, i.e.,
$(x,y)\in({\post})^{-1}\circ{\lhds}$.  The operator identity is again
Lemma \ref{lem:toolbox}(a).
\end{proof}

\begin{remark}\label{rem:strongps-reduct}
By Lemma \ref{lem:reduct}(3), every reduct frame $\rho\circ{\le}$ is
prefactoring.  The constructions $G_1$ and $G_2$ of
Sections \ref{sec:os4} and \ref{sec:itb} may be seen as converses of this
observation: a fundamental frame arises as a reduct of an
$\OS$-frame (resp.\ an FSTB-frame) \emph{over the same carrier} as soon
as it is balanced (resp.\ strongly factoring).
\end{remark}

\subsection{The canonical model satisfies the factoring conditions}
\label{subsec:canonical-strongps}

We now return to the canonical model of Section \ref{sec:canonical}.
Recall the notation: points of $W_d$ are pairs $\langle\Gamma,\Delta\rangle$
of a $\vdash$-consistent, $\vdash$-closed theory $\Gamma$ and a
counter-theory $\Delta$ containing $\bot$ and closed under
$\vdash$-predecessors, $\vee$, and $F_{\Gamma}$, with
$\Gamma\cap\Delta=\varnothing$; the openness relation is
\[
\langle\Gamma_2,\Delta_2\rangle \lhd_d \langle\Gamma_1,\Delta_1\rangle
\iff \Gamma_1\cap\Delta_2=\varnothing;
\]
$W_a$ consists of the pairs $\langle\Gamma,\dwn(\neg[\Gamma])\rangle$ for
$\Gamma$ consistent and closed, and $W_b$ of the pairs
$\langle\Th(\{\alpha\}),\dwn(\{\alpha\to\beta\})\rangle$ with
$\alpha\nvdash\alpha\to\beta$.  For a point
$x=\langle\Gamma_x,\Delta_x\rangle$ we refer to $\Gamma_x$ and $\Delta_x$
as its theory and counter-theory.  Two elementary observations will be
used repeatedly.

\begin{fact}\label{fact:canonical-refinement}
Let $W\subseteq W_d$ and let $x,y\in W$.  In the frame $(W,\lhd_d)$:
\begin{enumerate}
\item if $\Gamma_x\subseteq\Gamma_y$, then $y\pre x$;
\item if $\Delta_x\subseteq\Delta_y$, then $y\post x$;
\item $\neg[\Gamma_x]\subseteq\Delta_x$, and consequently
$\dwn(\neg[\Gamma_x])\subseteq\Delta_x$.
\end{enumerate}
\end{fact}

\begin{proof}
(1) If $z\lhd_d y$, then $\Gamma_y\cap\Delta_z=\varnothing$, hence
$\Gamma_x\cap\Delta_z=\varnothing$ by $\Gamma_x\subseteq\Gamma_y$, i.e.,
$z\lhd_d x$.

(2) If $y\lhd_d z$, then $\Gamma_z\cap\Delta_y=\varnothing$, hence
$\Gamma_z\cap\Delta_x=\varnothing$ by $\Delta_x\subseteq\Delta_y$, i.e.,
$x\lhd_d z$.

(3) Let $\alpha\in\Gamma_x$.  Since $\bot\in\Delta_x$,
$\alpha\wedge\bot\vdash\bot$, and $\Delta_x$ is closed under
$\vdash$-predecessors, we have $\alpha\wedge\bot\in\Delta_x$; hence
$\neg\alpha=\alpha\to\bot\in F_{\Gamma_x}(\Delta_x)\subseteq\Delta_x$.
Finally, if $\beta\vdash\neg\alpha$ for some $\alpha\in\Gamma_x$, then
$\beta\in\Delta_x$ by closure under $\vdash$-predecessors again.
\end{proof}

The carrier we will use adds to $W_a\cup W_b$ the following points, whose
theories are as small as possible and whose counter-theories are
inherited from $W_d$.

\begin{lemma}\label{lem:We-cont}
Let $\Thm:=\{\alpha\in\Fo\mid\ \vdash\alpha\}$ and
\[
W_e \;:=\; \{\langle\Thm,\Psi\rangle \mid
\langle\Phi,\Psi\rangle\in W_d \text{ for some } \Phi\}.
\]
If $\vdash$ satisfies the basic rules, then $W_e\subseteq W_d$.
\end{lemma}

\begin{proof}
Let $\langle\Thm,\Psi\rangle\in W_e$, witnessed by
$\langle\Phi,\Psi\rangle\in W_d$.  First, $\Thm\subseteq\Phi$: if
$\vdash\alpha$, then $\Phi\vdash\alpha$ by (MON), so $\alpha\in\Phi$
since $\Phi$ is $\vdash$-closed.  Hence $\Thm$ is $\vdash$-consistent
(as $\Phi$ is, using (MON)) and $\Thm\cap\Psi\subseteq\Phi\cap\Psi=\varnothing$.
Second, $\Thm$ is $\vdash$-closed: if $\Thm\vdash\alpha$, then by
finitarity (Remark \ref{rem:mon-fin}) there are theorems
$\gamma_1,\dots,\gamma_n$ with $\{\gamma_1,\dots,\gamma_n\}\vdash\alpha$;
cutting with $\vdash\gamma_i$ for each $i$ yields $\vdash\alpha$, so
$\alpha\in\Thm$.  Third, $\Psi$ contains $\bot$ and is closed under
$\vdash$-predecessors and $\vee$, since these conditions do not depend
on the theory component; and $\Psi$ is closed under $F_{\Thm}$ because
$F_{\Thm}(\Psi)\subseteq F_{\Phi}(\Psi)\subseteq\Psi$, using
$\Thm\subseteq\Phi$ and the monotonicity of $\Gamma\mapsto F_{\Gamma}(X)$.
\end{proof}

\begin{theorem}\label{thm:canonical-strongps}
Suppose $\vdash$ satisfies all the axioms of $\sysF$.  Let
$W_c:=W_a\cup W_b\cup W_e$ and let $\lhd_c$ be the restriction of
$\lhd_d$ to $W_c$.  Then $(W_c,\lhd_c)$ is reflexive, strongly factoring and converse-strongly factoring; in particular, it is balanced.
\end{theorem}

\begin{proof}
Throughout, $x=\langle\Gamma_x,\Delta_x\rangle$ and
$y=\langle\Gamma_y,\Delta_y\rangle$ range over $W_c$.  Note that
$W_c\subseteq W_d$ by Lemmas \ref{lem:least-ct}, \ref{lem:Wb-in-Wd}, and \ref{lem:We-cont}, so
Fact \ref{fact:canonical-refinement} applies, and $\lhd_c$ is reflexive
since $\Gamma_x\cap\Delta_x=\varnothing$ for all $x\in W_d$.

\emph{Strong factoring.}  Suppose $y\lhd_c x$, i.e.,
$\Gamma_x\cap\Delta_y=\varnothing$.  Let
\[
z\;:=\;\bigl\langle\,\Gamma_x,\ \dwn(\neg[\Gamma_x])\,\bigr\rangle .
\]
Since $\Gamma_x$ is consistent and closed (as $x\in W_d$), we have
$z\in W_a\subseteq W_c$.  We check the four required properties.
\begin{itemize}
\item $y\lhd_c z$: $\Gamma_z\cap\Delta_y=\Gamma_x\cap\Delta_y=\varnothing$
by hypothesis.
\item $z\lhd_c y$: suppose toward a contradiction that some
$\varphi\in\Gamma_y\cap\dwn(\neg[\Gamma_x])$, say $\varphi\vdash\neg\alpha$
with $\alpha\in\Gamma_x$.  Since $\Gamma_y$ is $\vdash$-closed,
$\neg\alpha\in\Gamma_y$, so $\neg\neg\alpha\in\neg[\Gamma_y]\subseteq\Delta_y$
by Fact \ref{fact:canonical-refinement}(3).  On the other hand, by
$(\neg\neg\mathrm{I})$ and closure of $\Gamma_x$ we get
$\neg\neg\alpha\in\Gamma_x$, contradicting
$\Gamma_x\cap\Delta_y=\varnothing$.  Hence
$\Gamma_y\cap\Delta_z=\varnothing$, i.e., $z\lhd_c y$.  Together with the
previous item, $z\lhds y$.
\item $z\pre x$: $\Gamma_x\subseteq\Gamma_z$ (indeed
$\Gamma_z=\Gamma_x$), so Fact \ref{fact:canonical-refinement}(1) applies.
\item $x\post z$: $\Delta_z=\dwn(\neg[\Gamma_x])\subseteq\Delta_x$ by
Fact \ref{fact:canonical-refinement}(3), so Fact
\ref{fact:canonical-refinement}(2) applies (with the roles
$\Delta_z\subseteq\Delta_x$).
\end{itemize}
Thus $z\lhds y$ and $z\prep x$, as required.

\emph{Converse-strong factoring.}  Suppose $x\lhd_c y$, i.e.,
$\Gamma_y\cap\Delta_x=\varnothing$.  Let
\[
z\;:=\;\langle\,\Thm,\ \Delta_x\,\rangle .
\]
Since $x\in W_d$, we have $z\in W_e\subseteq W_c$ by Lemma
\ref{lem:We-cont}.  We check:
\begin{itemize}
\item $z\lhd_c y$: $\Gamma_y\cap\Delta_z=\Gamma_y\cap\Delta_x=\varnothing$
by hypothesis.
\item $y\lhd_c z$: $\Gamma_z\cap\Delta_y=\Thm\cap\Delta_y
\subseteq\Gamma_y\cap\Delta_y=\varnothing$, using $\Thm\subseteq\Gamma_y$
(as in the proof of Lemma \ref{lem:We-cont}).  Together, $z\lhds y$.
\item $z\post x$: $\Delta_x\subseteq\Delta_z$ (indeed
$\Delta_z=\Delta_x$), so Fact \ref{fact:canonical-refinement}(2) applies.
\item $x\pre z$: $\Gamma_z=\Thm\subseteq\Gamma_x$, so Fact
\ref{fact:canonical-refinement}(1) applies.
\end{itemize}
Thus $z\lhds y$, $z\post x$, and $x\pre z$, i.e., $x\prep z$, which is
exactly the witness required for converse-strong factoring.
\end{proof}

\begin{remark}
The witness for strong factoring sits in $W_a$ and refines
$x$ by \emph{maximizing} the counter-theory relative to the theory
$\Gamma_x$ (down to the least admissible one, $\dwn(\neg[\Gamma_x])$,
which prunes $\Delta_x$); the witness for the inverse condition sits in
$W_e$ and refines $x$ in the converse sense by \emph{minimizing} the
theory (down to $\Thm$) while keeping the counter-theory $\Delta_x$.
The two constructions are exact mirror images, and only the first
requires the characteristic axiom $(\neg\neg\mathrm{I})$ of $\sysF$.
\end{remark}

We package the countermodel construction in the form used in the
faithfulness proofs below.

\begin{corollary}\label{cor:strong-countermodel}
Suppose $\vdash$ satisfies all the axioms of $\sysF$ and
$\varphi\nvdash_{\sysF}\psi$.  Then there are an $\sysF$-model
$\mathcal M=(W,\lhd,V)$ (i.e., a reflexive, pseudo-symmetric model whose
valuation takes values in $\FP(c_{\lhd})$) and a point $w\in W$ such
that:
\begin{enumerate}
\item $(W,\lhd)$ is strongly factoring and converse-strongly factoring (hence balanced);
\item $w\in\den{\varphi}_{\mathcal M}$ and
$w\notin\den{\psi}_{\mathcal M}$.
\end{enumerate}
\end{corollary}

\begin{proof}
Note first that $\varphi\nvdash\bot$: otherwise $\varphi\vdash\psi$ by
$(\bot\mathrm{E})$ and (CUT).  Hence
$\Th(\{\varphi\})$ is consistent and closed, so
\[
w\;:=\;\bigl\langle\,\Th(\{\varphi\}),\ \dwn\bigl(\neg[\Th(\{\varphi\})]\bigr)\,\bigr\rangle
\;\in\;W_a .
\]
Let $W_c:=W_a\cup W_b\cup W_e$ and let $\mathcal M_c$ be the canonical
model over $W_c$ as in Theorem \ref{thm:completeness} (whose hypotheses
$W_a\cup W_b\subseteq W_c\subseteq W_d$ hold by Lemmas \ref{lem:least-ct}, \ref{lem:Wb-in-Wd}
and \ref{lem:We-cont}).  By Theorem \ref{thm:completeness}, the Truth
Lemma holds over $W_c$, every $\den{\alpha}_{\mathcal M_c}$ is a
proposition (Fixpoint Lemma), and $\lhd_c$ is reflexive and
pseudo-symmetric; by Theorem \ref{thm:canonical-strongps}, $(W_c,\lhd_c)$
moreover satisfies the factoring conditions in (1).  Finally,
$\varphi\in\Th(\{\varphi\})$ while $\psi\notin\Th(\{\varphi\})$ (else
$\varphi\vdash\psi$), so the Truth Lemma yields (2).
\end{proof}

\section{A full and faithful translation into ortho-S4}\label{sec:os4}

\subsection{The logic $\OS$ and its relational semantics}

Let $\Lb$ be the language given by the grammar
\[
\varphi ::= \top \mid p \mid \neg\varphi \mid (\varphi\wedge\varphi)
\mid (\varphi\vee\varphi) \mid \Box\varphi ,
\]
where $p\in\mathrm{Prop}$, with $\bot:=\neg\top$.  Note that on the
target side negation is \emph{primitive}, as in orthologic, whereas in
$\Fo$ it is defined from the preconditional.

\begin{definition}\label{def:os4}
Ortho-S4, denoted $\OS$, is the smallest binary relation
${\vdash_{\OS}}\subseteq\Lb\times\Lb$ satisfying, for all
$\varphi,\psi,\chi\in\Lb$:
\begin{enumerate}
\item[(o1)] $\varphi\vdash\varphi$;\quad
$\varphi\vdash\top$;\quad
if $\varphi\vdash\psi$ and $\psi\vdash\chi$, then $\varphi\vdash\chi$;
\item[(o2)] $\varphi\wedge\psi\vdash\varphi$;\quad
$\varphi\wedge\psi\vdash\psi$;\quad
if $\varphi\vdash\psi$ and $\varphi\vdash\chi$, then
$\varphi\vdash\psi\wedge\chi$;
\item[(o3)] $\varphi\vdash\varphi\vee\psi$;\quad
$\psi\vdash\varphi\vee\psi$;\quad
if $\varphi\vdash\chi$ and $\psi\vdash\chi$, then
$\varphi\vee\psi\vdash\chi$;
\item[(o4)] $\varphi\vdash\neg\neg\varphi$;\quad
$\neg\neg\varphi\vdash\varphi$;\quad
$\varphi\wedge\neg\varphi\vdash\psi$;\quad
if $\varphi\vdash\psi$, then $\neg\psi\vdash\neg\varphi$;
\item[(o5)] if $\varphi\vdash\psi$, then $\Box\varphi\vdash\Box\psi$;\quad
$\Box\varphi\wedge\Box\psi\vdash\Box(\varphi\wedge\psi)$;\quad
$\top\vdash\Box\top$;
\item[(o6)] $\Box\varphi\vdash\varphi$;\quad
$\Box\varphi\vdash\Box\Box\varphi$.
\end{enumerate}
\end{definition}

Thus $\OS$ is orthologic \cite{Goldblatt1974,Holliday2023} equipped with
an S4 necessity.

For the semantics, fix a set $X$, a reflexive symmetric relation
$\between$ on $X$ (\emph{compatibility}), and write
$\neg_{\between}U:=\Box_{\between}(-U)$ and
$c_{\between}:=\neg_{\between}\neg_{\between}=\Box_{\between}\Diamond_{\between}$.
Viewing $\between$ itself as an openness relation, we take the
\emph{propositions} of $(X,\between)$ to be the members of
$\FP(c_{\between})$.  Indeed, $c_{\between}$ is exactly the closure
operator $c_{\lhd}$ of Lemma \ref{lem:toolbox}(g) for
$\lhd:=\between$, so
\[
\FP(c_{\between})=\Box_{\between}[\wp(X)],
\]
and $c_{\between}(U)$ is the least proposition of $(X,\between)$
containing $U$.

\begin{definition}\label{def:os4-frame}
An \emph{$\OS$-frame} is a triple $(X,\between,\le)$ where $X$ is a
nonempty set, $\between$ is a reflexive symmetric relation on $X$, $\le$
is a preorder on $X$, and
\begin{itemize}
\item[(int)] if $U\in\FP(c_{\between})$, then
$\Box_{\le}U\in\FP(c_{\between})$---equivalently, every set of the
form $\Box_{\le}\Box_{\between}Z$ belongs to $\FP(c_{\between})$.
\end{itemize}
An \emph{$\OS$-model} $\mathcal M=(X,\between,\le,V)$ adds a valuation
$V:\mathrm{Prop}\to\FP(c_{\between})$.  Truth sets are defined by
\[
\den{\top}=X,\quad \den{p}=V(p),\quad
\den{\neg\varphi}=\neg_{\between}\den{\varphi},\quad
\den{\varphi\wedge\psi}=\den{\varphi}\cap\den{\psi},
\]
\[
\den{\varphi\vee\psi}=c_{\between}\bigl(\den{\varphi}\cup\den{\psi}\bigr),
\qquad
\den{\Box\varphi}=\Box_{\le}\den{\varphi}.
\]
\end{definition}

We read $y\le x$ as ``$y$ refines $x$.'' Thus $\le$ is also written in predecessor style: $\Box_{\le}$ quantifies over refinements $y$ of $x$; in ordinary forward accessibility notation, this would be written $xRy$. This matches the conventions of \S\ref{subsec:toolbox}.
Condition (int) is the algebraic form of the modal-frame
interaction condition in \cite[Def.~2.16, Lem.~2.17]{HM2026}: it says
exactly that $\Box_{\le}$ sends $\between$-propositions to
$\between$-propositions.  We use this equivalent form throughout.

\begin{lemma}\label{lem:os4-dual}
Let $(X,\between,\le)$ be an $\OS$-frame.
\begin{enumerate}
\item $(\FP(c_{\between}),\cap,\sqcup,\neg_{\between},\varnothing,X)$ is a
(complete) ortholattice, where $U\sqcup V:=c_{\between}(U\cup V)$.
\item $\Box_{\le}$ restricts to an operation on $\FP(c_{\between})$
satisfying $\Box_{\le}X=X$,
$\Box_{\le}(U\cap V)=\Box_{\le}U\cap\Box_{\le}V$,
$\Box_{\le}U\subseteq U$, and
$\Box_{\le}U\subseteq\Box_{\le}\Box_{\le}U$.
\item Consequently, in every $\OS$-model,
$\den{\varphi}\in\FP(c_{\between})$ for all $\varphi\in\Lb$.
\end{enumerate}
\end{lemma}

\begin{proof}
(1) is a classical observation going back to Birkhoff (see \cite{Birkhoff1940},
\cite[\S~2.3]{HM2026} and \cite[\S~4]{Holliday2023} for discussion): the
members of $\FP(c_{\between})$ are the fixpoints of the closure operator
$c_{\between}$, hence are closed under intersections, with join given by
the closure of the union; $\neg_{\between}$ maps these fixpoints to
fixpoints since its image consists of fixpoints (Lemma
\ref{lem:toolbox}(g)), it is antitone,
$\neg_{\between}\neg_{\between}U=U$ on fixpoints by definition, and
$U\cap\neg_{\between}U=\varnothing$ by reflexivity of $\between$;
finally $\varnothing,X\in\FP(c_{\between})$ by reflexivity of $\between$.

(2) Closure under $\Box_{\le}$ is condition (int); the remaining properties
are Lemma \ref{lem:toolbox}(f),(c),(d).

(3) By induction on $\varphi$, using (1) for the connective clauses and
(2) for $\Box$.
\end{proof}




\begin{theorem}[{\cite[Thm.~3.3]{HM2026}}]\label{thm:os4-complete}
For all $\varphi,\psi\in\Lb$: $\varphi\vdash_{\OS}\psi$ if and only if
$\den{\varphi}\subseteq\den{\psi}$ in every $\OS$-model.
\end{theorem}

Note that the frames employed in
the completeness proof of \cite{HM2026} satisfy condition (int) of
Definition \ref{def:os4-frame}, since (int) is precisely the requirement
that $\Box$ send propositions to propositions, which is built into the
modal frames of \cite[Def.~2.16]{HM2026}.

The key observation on the target side is that the Sasaki hook is
definable in $\Lb$ and that its truth sets take exactly the form of the
consequent of the relational preconditional clause.

\begin{lemma}[Sasaki hook]\label{lem:sasaki}
For $\varphi,\psi\in\Lb$, define
\[
\varphi\shook\psi\;:=\;\neg\bigl(\varphi\wedge\neg(\varphi\wedge\psi)\bigr).
\]
Then in any $\OS$-model, with $A=\den{\varphi}$ and $B=\den{\psi}$,
\[
\den{\varphi\shook\psi}
\;=\;\Box_{\between}\bigl(-A\cup\Diamond_{\between}(A\cap B)\bigr).
\]
\end{lemma}

\begin{proof}
Unwinding the clauses and using
$-\Box_{\between}(-W)=\Diamond_{\between}W$:
\[
\den{\varphi\shook\psi}
=\Box_{\between}\Bigl(-\bigl(A\cap\Box_{\between}(-(A\cap B))\bigr)\Bigr)
=\Box_{\between}\bigl(-A\cup-\Box_{\between}(-(A\cap B))\bigr)
=\Box_{\between}\bigl(-A\cup\Diamond_{\between}(A\cap B)\bigr).\qedhere
\]
\end{proof}

\begin{remark}
On the dual ortholattice $\FP(c_{\between})$, the operation
$(A,B)\mapsto\neg_{\between}(A\cap\neg_{\between}(A\cap B))$ is
precisely the Sasaki hook, which is a preconditional by
\cite[Cor.~1]{Holliday2025}.  This is the algebraic counterpart of the
translation defined next, and it connects the present section to the
taxonomy of preconditionals recalled in Section \ref{sec:algebra}.
\end{remark}

\subsection{The GMT-style translation}

\begin{definition}\label{def:transI}
The translation $(\cdot)\tI:\Fo\to\Lb$ is defined by
\[
\bot\tI=\bot,\qquad p\tI=\Box p,\qquad
(\alpha\wedge\beta)\tI=\alpha\tI\wedge\beta\tI,\qquad
(\alpha\vee\beta)\tI=\alpha\tI\vee\beta\tI,
\]
\[
(\alpha\to\beta)\tI\;=\;\Box\,(\alpha\tI\shook\beta\tI)
\;=\;\Box\neg\bigl(\alpha\tI\wedge\neg(\alpha\tI\wedge\beta\tI)\bigr).
\]
\end{definition}

\begin{remark}\label{rem:transI-neg}
The clause for $\to$ has the same shape $\Box(\cdot\to\cdot)$ as the
original GMT clause for intuitionistic implication, with the material
arrow replaced by the Sasaki hook.  Moreover, for the defined negation
$\neg\alpha:=\alpha\to\bot$ of $\Fo$ we have, in every $\OS$-model,
\[
\den{(\neg\alpha)\tI}
=\Box_{\le}\Box_{\between}\bigl(-\den{\alpha\tI}\cup
\Diamond_{\between}(\den{\alpha\tI}\cap\varnothing)\bigr)
=\Box_{\le}\,\neg_{\between}\den{\alpha\tI}
=\den{\Box\neg\alpha\tI},
\]
using Lemma \ref{lem:sasaki} and $\den{\bot}=\neg_{\between}X=\varnothing$
(reflexivity of $\between$).  Hence, by Theorem \ref{thm:os4-complete}, $(\neg\alpha)\tI$ and
$\Box\neg\alpha\tI$ are interderivable in $\OS$, and an easy induction
shows that on the $\{\wedge,\vee,\neg\}$-fragment of $\Fo$ (embedded via
$\neg\chi:=\chi\to\bot$) the translation $(\cdot)\tI$ agrees, up to
$\OS$-interderivability, with the GMT-style translation of fundamental
logic into $\OS$ from \cite[Thm.~1.1]{HM2026}.
\end{remark}

\subsection{From $\OS$-models to $\sysF$-models}

\begin{definition}\label{def:F1}
For an $\OS$-model $\mathcal M=(X,\between,\le,V)$, let
\[
F_1(\mathcal M)\;:=\;(X,\ \lhd,\ V^{\flat}),
\qquad\text{where }\lhd:={\between}\circ{\le}
\ \text{ and }\ V^{\flat}(p):=\Box_{\le}V(p).
\]
Explicitly, $y\lhd x$ iff there is $u$ with $y\between u$ and $u\le x$.
\end{definition}

The frame component of $F_1$ is exactly the fundamental reduct of
\cite[Def.~3.4 and Def.~3.6]{HM2026}: the relation
$\lhd=\between\circ\le$ is HM's induced openness relation.  

\begin{lemma}\label{lem:F1-frame}
Let $\mathcal M=(X,\between,\le,V)$ be an $\OS$-model and
$\lhd={\between}\circ{\le}$.
\begin{enumerate}
\item $(X,\lhd)$ is reflexive and (prefactoring) pseudo-symmetric.
\item $\FP(c_{\lhd})=\Box_{\le}\bigl[\FP(c_{\between})\bigr]$, and every
$A\in\FP(c_{\lhd})$ also belongs to $\FP(c_{\between})$ and satisfies
$\Box_{\le}A=A$.
\item For all $A,B\in\FP(c_{\lhd})$:
\[
\Diamond_{\rhd}(A\cap B)=\Diamond_{\between}(A\cap B),\qquad
c_{\lhd}(A\cup B)=c_{\between}(A\cup B),
\]
\[
\to_{\lhd}(A,B)\;=\;\Box_{\le}\,\Box_{\between}\bigl(-A\cup
\Diamond_{\between}(A\cap B)\bigr).
\]
\end{enumerate}
\end{lemma}

\begin{proof}
Lemma \ref{lem:reduct}, with $\rho:=\between$, gives (1), the range
identity and $\Box_{\le}$-fixedness in (2), and all the identities in
(3) except the unboxed closure identity.  The only direct use of (int)
is the remaining fixpoint claim: by Lemma \ref{lem:reduct}(4), every
$A\in\FP(c_{\lhd})$ has the form $\Box_{\le}U$ for some
$U\in\FP(c_{\between})$, so (int) gives $A\in\FP(c_{\between})$.

It remains to remove the outer $\Box_{\le}$ from the closure formula in
Lemma \ref{lem:reduct}(6).  The set $c_{\lhd}(A\cup B)$ is in
$\FP(c_{\lhd})$, hence belongs to $\FP(c_{\between})$ by the preceding
paragraph;
as it contains $A\cup B$, minimality of $c_{\between}(A\cup B)$ gives
\[
c_{\between}(A\cup B)\subseteq c_{\lhd}(A\cup B).
\]
Conversely, Lemma \ref{lem:reduct}(6) and reflexivity of $\le$ give
$c_{\lhd}(A\cup B)=\Box_{\le}c_{\between}(A\cup B)
\subseteq c_{\between}(A\cup B)$.
\end{proof}

\begin{theorem}\label{thm:F1}
Let $\mathcal M=(X,\between,\le,V)$ be an $\OS$-model.  Then:
\begin{enumerate}
\item $F_1(\mathcal M)$ is an $\sysF$-model;
\item for all $\varphi\in\Fo$,
$\den{\varphi}_{F_1(\mathcal M)}=\den{\varphi\tI}_{\mathcal M}$.
\end{enumerate}
\end{theorem}

\begin{proof}
(1) The frame $(X,\lhd)$ is reflexive and pseudo-symmetric by Lemma
\ref{lem:F1-frame}(1), and $V^{\flat}(p)=\Box_{\le}V(p)\in
\Box_{\le}[\FP(c_{\between})]=\FP(c_{\lhd})$ by Lemma
\ref{lem:F1-frame}(2), since $V(p)\in\FP(c_{\between})$.

(2) By induction on $\varphi$.  Throughout, let
$A:=\den{\alpha}_{F_1(\mathcal M)}$,
$B:=\den{\beta}_{F_1(\mathcal M)}$ and
$A':=\den{\alpha\tI}_{\mathcal M}$,
$B':=\den{\beta\tI}_{\mathcal M}$; the inductive hypotheses are $A=A'$
and $B=B'$, and note that $A,B\in\FP(c_{\lhd})$ since truth sets in a
model are propositions.

\emph{Base cases.}  $\den{p}_{F_1(\mathcal M)}=V^{\flat}(p)
=\Box_{\le}V(p)=\den{\Box p}_{\mathcal M}=\den{p\tI}_{\mathcal M}$.
For $\bot$: $\den{\bot}_{F_1(\mathcal M)}$ is the set of
$\lhd$-absurd states, which is empty since $\lhd$ is reflexive; and
$\den{\bot\tI}_{\mathcal M}=\den{\neg\top}_{\mathcal M}
=\neg_{\between}X=\varnothing$ by reflexivity of $\between$.

\emph{Conjunction.}  Both semantics interpret $\wedge$ by intersection,
so the claim follows from the inductive hypotheses.

\emph{Disjunction.}  By Lemma \ref{lem:F1-frame}(3) and the inductive
hypotheses,
\[
\den{\alpha\vee\beta}_{F_1(\mathcal M)}
=c_{\lhd}(A\cup B)
=c_{\between}(A'\cup B')
=\den{\alpha\tI\vee\beta\tI}_{\mathcal M}
=\den{(\alpha\vee\beta)\tI}_{\mathcal M}.
\]

\emph{Conditional.}  By Lemma \ref{lem:F1-frame}(3), the inductive
hypotheses, and Lemma \ref{lem:sasaki},
\[
\den{\alpha\to\beta}_{F_1(\mathcal M)}
=\to_{\lhd}(A,B)
=\Box_{\le}\Box_{\between}\bigl(-A'\cup\Diamond_{\between}(A'\cap B')\bigr)
=\Box_{\le}\den{\alpha\tI\shook\beta\tI}_{\mathcal M}
=\den{(\alpha\to\beta)\tI}_{\mathcal M}.\qedhere
\]
\end{proof}


\subsection{From $\sysF$-models to $\OS$-models}

\begin{definition}\label{def:G1}
For an $\sysF$-model $\mathcal N=(X,\lhd,V)$, let
\[
G_1(\mathcal N)\;:=\;(X,\ \lhds,\ \pre,\ V),
\]
i.e., the compatibility relation is the symmetric kernel of the openness
relation and the S4 preorder is prerefinement. When $(X,\lhd)$ is
balanced, this frame component is exactly the orthomodal companion of
\cite[Def.~3.13 and Def.~3.15]{HM2026}, with local notation $\lhds$ for
HM's symmetric kernel and $\pre$ for HM's prerefinement relation.
\end{definition}

\begin{theorem}\label{thm:G1}
Let $\mathcal N=(X,\lhd,V)$ be an $\sysF$-model that is balanced.  Then:
\begin{enumerate}
\item $G_1(\mathcal N)$ is an $\OS$-model;
\item $F_1(G_1(\mathcal N))=\mathcal N$;
\item for all $\varphi\in\Fo$,
$\den{\varphi}_{\mathcal N}=\den{\varphi\tI}_{G_1(\mathcal N)}$.
\end{enumerate}
\end{theorem}

\begin{proof}
(1) Since $(X,\lhd)$ is balanced, \cite[Lem.~3.14]{HM2026} shows that
$(X,\lhds,\pre)$ is an $\OS$-frame (using the equivalent formulation
(int) in Definition \ref{def:os4-frame}).  Moreover,
\cite[Lem.~3.16(1)]{HM2026} gives
$\FP(c_{\lhd})\subseteq\FP(c_{\lhds})$, so $V$ is a legitimate
$\OS$-valuation.

(2) For the frames, the openness relation of
$F_1(G_1(\mathcal N))$ is ${\lhds}\circ{\pre}$, which equals $\lhd$ by
prefactoring (Proposition \ref{prop:strongps-char}(1)).
For the valuations, choose $Z\subseteq X$ with $V(p)=\Box_{\lhd}Z$ by
Lemma \ref{lem:toolbox}(g).  Proposition \ref{prop:strongps-char}(1) gives
$V(p)=\Box_{\pre}\Box_{\lhds}Z$, so Lemma \ref{lem:toolbox}(d) yields
\[
\Box_{\pre}V(p)=\Box_{\pre}\Box_{\pre}\Box_{\lhds}Z
=\Box_{\pre}\Box_{\lhds}Z=V(p),
\]
so the valuation of $F_1(G_1(\mathcal N))$ is again $V$.

(3) Applying Theorem \ref{thm:F1} to the $\OS$-model
$G_1(\mathcal N)$ yields
$\den{\varphi}_{F_1(G_1(\mathcal N))}=\den{\varphi\tI}_{G_1(\mathcal N)}$
for all $\varphi\in\Fo$; by (2) the left-hand side is
$\den{\varphi}_{\mathcal N}$, since truth sets are determined by the
frame and the valuation.
\end{proof}

\subsection{The main theorem for $\OS$}

\begin{theorem}\label{thm:main-os4}
For all $\varphi,\psi\in\Fo$:
\[
\varphi\vdash_{\sysF}\psi
\iff
\varphi\tI\vdash_{\OS}\psi\tI .
\]
\end{theorem}

\begin{proof}
($\Rightarrow$)  Suppose $\varphi\vdash_{\sysF}\psi$, and let
$\mathcal M$ be any $\OS$-model.  By Theorem \ref{thm:F1}(1),
$F_1(\mathcal M)$ is an $\sysF$-model, so by the soundness of $\sysF$
with respect to $\sysF$-models (Theorem \ref{thm:soundness}),
$\den{\varphi}_{F_1(\mathcal M)}\subseteq\den{\psi}_{F_1(\mathcal M)}$.
By Theorem \ref{thm:F1}(2), this says
$\den{\varphi\tI}_{\mathcal M}\subseteq\den{\psi\tI}_{\mathcal M}$.
Since $\mathcal M$ was arbitrary, the completeness of $\OS$ (Theorem
\ref{thm:os4-complete}) yields $\varphi\tI\vdash_{\OS}\psi\tI$.

($\Leftarrow$)  Suppose $\varphi\nvdash_{\sysF}\psi$.  By Corollary
\ref{cor:strong-countermodel}, there are an $\sysF$-model $\mathcal M$
that is balanced and a point $w$ with
$w\in\den{\varphi}_{\mathcal M}$ and
$w\notin\den{\psi}_{\mathcal M}$.  By Theorem \ref{thm:G1},
$G_1(\mathcal M)$ is an $\OS$-model with
$\den{\varphi\tI}_{G_1(\mathcal M)}=\den{\varphi}_{\mathcal M}\ni w$ and
$\den{\psi\tI}_{G_1(\mathcal M)}=\den{\psi}_{\mathcal M}\not\ni w$.
Hence $\den{\varphi\tI}\not\subseteq\den{\psi\tI}$ in some $\OS$-model,
so $\varphi\tI\nvdash_{\OS}\psi\tI$ by the left-to-right
direction of Theorem~\ref{thm:os4-complete}.
\end{proof}

\begin{corollary}\label{cor:main-os4-sets}
For any $\Gamma\subseteq\Fo$ and $\varphi\in\Fo$:
$\Gamma\vdash_{\sysF}\varphi$ iff
$(\gamma_1\wedge\dots\wedge\gamma_n)\tI\vdash_{\OS}\varphi\tI$ for some
$\gamma_1,\dots,\gamma_n\in\Gamma$ (where the empty conjunction is
$\top$).
\end{corollary}

\begin{proof}
Immediate from Theorem \ref{thm:main-os4} and the finitarity of
$\vdash_{\sysF}$ (Remark \ref{rem:mon-fin}).
\end{proof}

\begin{corollary}\label{cor:hm11}
For all $\varphi,\psi$ in the $\{\wedge,\vee,\neg\}$-fragment
(embedded into $\Fo$ via $\neg\chi:=\chi\to\bot$):
\[
\varphi\vdash_{\mathcal F}\psi
\iff
\varphi\tI\vdash_{\OS}\psi\tI ,
\]
where $\vdash_{\mathcal F}$ is fundamental logic.  In view of Remark
\ref{rem:transI-neg}, this recovers Theorem 1.1 of \cite{HM2026}.
\end{corollary}

\begin{proof}
Combine the conservativity of $\sysF$ over fundamental logic
(Corollary \ref{cor:conservative}) with Theorem \ref{thm:main-os4}.
\end{proof}

\section{A full and faithful translation into intuitionistic KTB}
\label{sec:itb}

\subsection{The logic $\ITB$ and its relational semantics}

Let $\Lbd$ be the language given by the grammar
\[
\varphi ::= p \mid \bot \mid (\varphi\wedge\varphi) \mid
(\varphi\vee\varphi) \mid (\varphi\to\varphi) \mid \Box\varphi \mid
\Diamond\varphi ,
\]
with $\neg\varphi:=\varphi\to\bot$.  In this target intuitionistic
language only, we use $\top$ as the standard abbreviation for
$\bot\to\bot$.  Here the target conditional is the intuitionistic
implication, and---as in
intuitionistic modal logic in the style of Fischer Servi
\cite{FS1984}---$\Box$ and $\Diamond$ are both primitive and
not interdefinable.

\begin{definition}\label{def:itb}
Intuitionistic KTB, denoted $\ITB$, is the smallest binary relation
${\vdash_{\ITB}}\subseteq\Lbd\times\Lbd$ satisfying, for all
$\varphi,\psi,\chi,\alpha,\beta\in\Lbd$:
\begin{enumerate}
\item[(i1)] $\varphi\vdash\varphi$;\quad if $\varphi\vdash\psi$ and
$\psi\vdash\chi$, then $\varphi\vdash\chi$;\quad $\bot\vdash\varphi$;
\item[(i2)] $\varphi\wedge\psi\vdash\varphi$;\quad
$\varphi\wedge\psi\vdash\psi$;\quad if $\varphi\vdash\psi$ and
$\varphi\vdash\chi$, then $\varphi\vdash\psi\wedge\chi$;
\item[(i3)] $\varphi\vdash\varphi\vee\psi$;\quad
$\psi\vdash\varphi\vee\psi$;\quad if $\varphi\wedge\alpha\vdash\chi$ and
$\varphi\wedge\beta\vdash\chi$, then
$\varphi\wedge(\alpha\vee\beta)\vdash\chi$;
\item[(i4)] (residuation) $\varphi\wedge\psi\vdash\chi$ if and only if
$\varphi\vdash\psi\to\chi$;
\item[(i5)] if $\varphi\vdash\psi$, then $\Box\varphi\vdash\Box\psi$ and
$\Diamond\varphi\vdash\Diamond\psi$;\quad
$\Box\varphi\wedge\Box\psi\vdash\Box(\varphi\wedge\psi)$;\quad
$\top\vdash\Box\top$;\quad
$\Diamond(\varphi\vee\psi)\vdash\Diamond\varphi\vee\Diamond\psi$;\quad
$\Diamond\bot\vdash\bot$;
\item[(i6)] (Fischer Servi)
$\Diamond(\varphi\to\psi)\vdash\Box\varphi\to\Diamond\psi$;\quad
$(\Diamond\varphi\to\Box\psi)\vdash\Box(\varphi\to\psi)$;
\item[(i7)] (T) $\Box\varphi\vdash\varphi$;\quad
$\varphi\vdash\Diamond\varphi$;
\item[(i8)] (B) $\varphi\vdash\Box\Diamond\varphi$;\quad
$\Diamond\Box\varphi\vdash\varphi$.
\end{enumerate}
\end{definition}

\begin{definition}\label{def:fstb-frame}
An \emph{FSTB-frame} is a triple $(X,\le,\sim)$ where $X$ is a nonempty
set, $\le$ is a preorder on $X$, $\sim$ is a reflexive symmetric
relation on $X$, and the following interaction condition holds:
\begin{itemize}
\item[(fs)] for every $Z\subseteq X$,
$\Box_{\le}\Box_{\sim}Z\subseteq\Box_{\sim}\Box_{\le}Z$;
equivalently (by Lemma \ref{lem:toolbox}(a),(b)):
whenever $y\le u$ and $u\sim x$, there is $v$ with $y\sim v$ and
$v\le x$.
\end{itemize}
Call $U\subseteq X$ \emph{persistent} if $y\le x\in U$ implies $y\in U$;
equivalently, $\Box_{\le}U=U$; equivalently,
$\Diamond_{\le^{-1}}U=U$.  An \emph{FSTB-model}
$\mathcal M=(X,\le,\sim,V)$ adds a valuation $V$ taking values in the
persistent subsets of $X$.  Truth sets are defined by
\[
\den{\bot}=\varnothing,\quad \den{p}=V(p),\quad
\den{\varphi\wedge\psi}=\den{\varphi}\cap\den{\psi},\quad
\den{\varphi\vee\psi}=\den{\varphi}\cup\den{\psi},
\]
\[
\den{\varphi\to\psi}=\Box_{\le}\bigl(-\den{\varphi}\cup\den{\psi}\bigr),
\qquad
\den{\Diamond\varphi}=\Diamond_{\sim}\den{\varphi},
\qquad
\den{\Box\varphi}=\Box_{\le}\Box_{\sim}\den{\varphi}.
\]
\end{definition}

As before, $y\le x$ means that $y$ is a refinement/predecessor of $x$. Thus persistence is persistence under refinement, and $\to$ and $\Box$ quantify over refinements in the same predecessor-style sense.
Condition (fs) is the Fischer--Servi interaction condition in
\cite[Def.~4.3]{HM2026}.  Since $\sim$ is symmetric, it is equivalent
to the second interaction condition displayed there; we use the single
operator inclusion above as the frame condition.

\begin{lemma}\label{lem:itb-persistent}
In every FSTB-model, $\den{\varphi}$ is persistent for all
$\varphi\in\Lbd$.
\end{lemma}

\begin{proof}
By induction on $\varphi$.  Persistent sets contain $\varnothing$ and
are closed under $\cap$ and $\cup$; sets of the form $\Box_{\le}W$ are
persistent by Lemma \ref{lem:toolbox}(d), which covers the $\to$ and
$\Box$ clauses.  For $\Diamond$: suppose $U$ is persistent,
$x\in\Diamond_{\sim}U$, and $z\le x$; pick $y\sim x$ with $y\in U$.
Applying (fs) to $z\le x$ and $x\sim y$ gives $v$ with $z\sim v$ and
$v\le y$; then $v\in U$ by persistence, so $z\in\Diamond_{\sim}U$.
\end{proof}

\begin{theorem}[{\cite[Lem.~4.21]{HM2026}}]\label{thm:itb-complete}
For all $\varphi,\psi\in\Lbd$: $\varphi\vdash_{\ITB}\psi$ if and only if
$\den{\varphi}\subseteq\den{\psi}$ in every FSTB-model.
\end{theorem}


\subsection{The Goldblatt-style translation}

\begin{definition}\label{def:transO}
The translation $(\cdot)\tO:\Fo\to\Lbd$ is defined by
\[
\bot\tO=\bot,\qquad p\tO=\Box\Diamond p,\qquad
(\alpha\wedge\beta)\tO=\alpha\tO\wedge\beta\tO,\qquad
(\alpha\vee\beta)\tO=\Box\Diamond(\alpha\tO\vee\beta\tO),
\]
\[
(\alpha\to\beta)\tO\;=\;\Box\bigl(\alpha\tO\to
\Diamond(\alpha\tO\wedge\beta\tO)\bigr).
\]
\end{definition}

\begin{remark}\label{rem:transO-neg}
For the defined negation of $\Fo$ we compute, in any FSTB-model:
$\den{\alpha\tO\wedge\bot}=\varnothing$, so
$\den{\Diamond(\alpha\tO\wedge\bot)}=\varnothing$ and
\[
\den{(\neg\alpha)\tO}
=\Box_{\le}\Box_{\sim}\Box_{\le}\bigl(-\den{\alpha\tO}\bigr)
=\den{\Box\neg\alpha\tO}.
\]
Thus, on the $\{\wedge,\vee,\neg\}$-fragment, $(\cdot)\tO$ agrees (on
the nose, given the clause $(\neg\alpha)\tO=\Box\neg\alpha\tO$) with the
Goldblatt-style translation of fundamental logic into $\ITB$ from
\cite[Thm.~1.2]{HM2026}.
\end{remark}

\begin{remark}\label{rem:strict-sasaki}
In the involutive (orthologic) specialization of $\SysF$, the conditional is the
Sasaki hook $\neg(\alpha\wedge\neg(\alpha\wedge\beta))$.  On classical KTB-models,
its Goldblatt translation is classically equivalent to
\[
\Box\bigl(\alpha\tO\to\Diamond(\alpha\tO\wedge\beta\tO)\bigr).
\]
Thus the displayed strict clause is the natural intuitionistic-modal continuation of the Goldblatt
image of the Sasaki hook: the material arrow of classical KTB is replaced by the intuitionistic
arrow.
\end{remark}

\subsection{From FSTB-models to $\sysF$-models}

\begin{definition}\label{def:F2}
For an FSTB-model $\mathcal M=(X,\le,\sim,V)$, let
\[
F_2(\mathcal M)\;:=\;(X,\ \lhd,\ V^{\flat}),
\qquad\text{where }\lhd:={\sim}\circ{\le}
\ \text{ and }\ V^{\flat}(p):=\Box_{\le}\Box_{\sim}\Diamond_{\sim}V(p).
\]
Explicitly, $y\lhd x$ iff there is $u$ with $y\sim u$ and $u\le x$, and
$V^{\flat}(p)=\den{\Box\Diamond p}_{\mathcal M}$.
\end{definition}

The frame component of $F_2$ is exactly the fundamental reduct of
\cite[Def.~4.6]{HM2026}: the relation
$\lhd={\sim}\circ{\le}$ is HM's induced openness relation.  Note that it is literally the same construction
$\rho\circ{\le}$ as in Definition \ref{def:F1}, with the KTB relation
$\sim$ in place of the compatibility relation $\between$.

\begin{lemma}\label{lem:F2-frame}
Let $\mathcal M=(X,\le,\sim,V)$ be an FSTB-model and
$\lhd={\sim}\circ{\le}$.
\begin{enumerate}
\item $(X,\lhd)$ is reflexive and (prefactoring) pseudo-symmetric.
\item $\FP(c_{\lhd})=\Box_{\le}\Box_{\sim}[\wp(X)]$, and every
$A\in\FP(c_{\lhd})$ is persistent.
\item For every persistent $U$, $\Diamond_{\le^{-1}}U=U$; consequently,
for $A,B\in\FP(c_{\lhd})$,
\[
\Diamond_{\rhd}(A\cap B)=\Diamond_{\sim}(A\cap B)
\qquad\text{and}\qquad
\Diamond_{\rhd}(A\cup B)=\Diamond_{\sim}(A\cup B).
\]
\item For every $W\subseteq X$,
$\Box_{\le}\Box_{\sim}W=\Box_{\le}\Box_{\sim}\Box_{\le}W$.
\end{enumerate}
\end{lemma}

\begin{proof}
Items (1)--(3) follow from Lemma \ref{lem:reduct}(1),(3)--(6), with
$\rho:=\sim$; here $\Box_{\le}A=A$ is exactly persistence.  For (4),
the right-to-left inclusion follows from
$\Box_{\le}W\subseteq W$ (Lemma \ref{lem:toolbox}(c)) by monotonicity.
For left-to-right, by Lemma \ref{lem:toolbox}(d) and condition (fs),
\[
\Box_{\le}\Box_{\sim}W
=\Box_{\le}\bigl(\Box_{\le}\Box_{\sim}W\bigr)
\subseteq\Box_{\le}\bigl(\Box_{\sim}\Box_{\le}W\bigr).\qedhere
\]
\end{proof}

\begin{theorem}\label{thm:F2}
Let $\mathcal M=(X,\le,\sim,V)$ be an FSTB-model.  Then:
\begin{enumerate}
\item $F_2(\mathcal M)$ is an $\sysF$-model;
\item for all $\varphi\in\Fo$,
$\den{\varphi}_{F_2(\mathcal M)}=\den{\varphi\tO}_{\mathcal M}$.
\end{enumerate}
\end{theorem}

\begin{proof}
(1) The frame is reflexive and pseudo-symmetric by Lemma
\ref{lem:F2-frame}(1), and
$V^{\flat}(p)=\Box_{\le}\Box_{\sim}\bigl(\Diamond_{\sim}V(p)\bigr)
\in\Box_{\le}\Box_{\sim}[\wp(X)]=\FP(c_{\lhd})$.

(2) By induction on $\varphi$; as before, let
$A,B$ be the truth sets of $\alpha,\beta$ in $F_2(\mathcal M)$ and
$A',B'$ those of $\alpha\tO,\beta\tO$ in $\mathcal M$, with inductive
hypotheses $A=A'$, $B=B'$ and $A,B\in\FP(c_{\lhd})$.

\emph{Base cases.}  $\den{p}_{F_2(\mathcal M)}=V^{\flat}(p)
=\den{\Box\Diamond p}_{\mathcal M}=\den{p\tO}_{\mathcal M}$.  For
$\bot$: both sides are $\varnothing$, by reflexivity of $\lhd$ on the
left and by the semantic clause on the right.

\emph{Conjunction.}  Immediate.

\emph{Disjunction.}  Using Lemma \ref{lem:reduct}(2), Lemma
\ref{lem:F2-frame}(3), and the inductive hypotheses,
\[
\den{\alpha\vee\beta}_{F_2(\mathcal M)}
=c_{\lhd}(A\cup B)
=\Box_{\le}\Box_{\sim}\Diamond_{\sim}\Diamond_{\le^{-1}}(A\cup B)
=\Box_{\le}\Box_{\sim}\Diamond_{\sim}(A'\cup B')
=\den{\Box\Diamond(\alpha\tO\vee\beta\tO)}_{\mathcal M},
\]
which is $\den{(\alpha\vee\beta)\tO}_{\mathcal M}$.

\emph{Conditional.}  Using Lemma \ref{lem:reduct}(2) and Lemma
\ref{lem:F2-frame}(3),(4), and the inductive hypotheses,
\begin{align*}
\den{\alpha\to\beta}_{F_2(\mathcal M)}
&=\Box_{\lhd}\bigl(-A\cup\Diamond_{\rhd}(A\cap B)\bigr)
=\Box_{\le}\Box_{\sim}\bigl(-A'\cup\Diamond_{\sim}(A'\cap B')\bigr)\\
&=\Box_{\le}\Box_{\sim}\Box_{\le}\bigl(-A'\cup\Diamond_{\sim}(A'\cap B')\bigr)
=\den{\Box\bigl(\alpha\tO\to\Diamond(\alpha\tO\wedge\beta\tO)\bigr)}_{\mathcal M},
\end{align*}
which is $\den{(\alpha\to\beta)\tO}_{\mathcal M}$.
\end{proof}


\subsection{From $\sysF$-models to FSTB-models}

\begin{definition}\label{def:G2}
For an $\sysF$-model $\mathcal N=(X,\lhd,V)$, let
\[
G_2(\mathcal N)\;:=\;(X,\ \prep,\ \lhds,\ V^{\sharp}),
\qquad\text{where }
V^{\sharp}(p):=\Diamond_{(\prep)^{-1}}V(p)
=\{z\mid \exists x\,(z\prep x \text{ and } x\in V(p))\},
\]
i.e., the intuitionistic preorder is strong refinement, the KTB relation
is the symmetric kernel of openness, and the valuation is closed
downwards under strong refinement. When $(X,\lhd)$ is strongly factoring, this frame component is exactly
the FSTB companion of \cite[Def.~4.15]{HM2026}.
\end{definition}

\begin{theorem}\label{thm:G2}
Let $\mathcal N=(X,\lhd,V)$ be an $\sysF$-model that is strongly factoring.  Then:
\begin{enumerate}
\item $G_2(\mathcal N)$ is an FSTB-model;
\item $F_2(G_2(\mathcal N))=\mathcal N$;
\item for all $\varphi\in\Fo$,
$\den{\varphi}_{\mathcal N}=\den{\varphi\tO}_{G_2(\mathcal N)}$.
\end{enumerate}
\end{theorem}

\begin{proof}
By Proposition \ref{prop:strongps-char}(2), strong factoring gives, for all $Z,U\subseteq X$,
\[
\Box_{\lhd}Z=\Box_{\prep}\Box_{\lhds}Z
\tag{$*$}
\]
\[
\Diamond_{\rhd}U=\Diamond_{\lhds}\Diamond_{(\prep)^{-1}}U .
\tag{$**$}
\]

(1) By \cite[Lem.~4.14]{HM2026}, $(X,\prep,\lhds)$ is an FSTB-frame.
Moreover, $V^{\sharp}(p)$ is persistent with respect to $\prep$: if
$y\prep x$ and $x\prep w$ with $w\in V(p)$, then $y\prep w$ by
transitivity.

(2) For the frames: the openness relation of $F_2(G_2(\mathcal N))$ is
${\lhds}\circ{\prep}$, which equals $\lhd$ by strong factoring (Proposition \ref{prop:strongps-char}(2)).  For the
valuations: the valuation of $F_2(G_2(\mathcal N))$ at $p$ is
\[
\Box_{\prep}\Box_{\lhds}\Diamond_{\lhds}\,V^{\sharp}(p)
\;\overset{(*)}{=}\;
\Box_{\lhd}\,\Diamond_{\lhds}\Diamond_{(\prep)^{-1}}V(p)
\;\overset{(**)}{=}\;
\Box_{\lhd}\Diamond_{\rhd}V(p)
\;=\;c_{\lhd}\bigl(V(p)\bigr)
\;=\;V(p),
\]
since $V(p)\in\FP(c_{\lhd})$.

(3) Apply Theorem \ref{thm:F2} to the FSTB-model $G_2(\mathcal N)$ and
use (2), exactly as in the proof of Theorem \ref{thm:G1}(3).
\end{proof}

\begin{remark}\label{rem:fs-mirror}
Suppose one strengthens Definition \ref{def:fstb-frame} by the mirror
interaction condition
\[
\text{(fs$'$)}\qquad
\Box_{\sim}\Box_{\le}Z\subseteq\Box_{\le}\Box_{\sim}Z
\quad\text{for all }Z\subseteq X,
\]
equivalently: whenever $y\sim u$ and $u\le x$, there is $v$ with
$y\le v$ and $v\sim x$.  For $G_2(\mathcal N)$, condition (fs$'$)
amounts to ${\lhds}\circ{\prep}\subseteq{\prep}\circ{\lhds}$, i.e., to
the statement that whenever $y\lhd x$ there is $v$ with $y\prep v$ and
$v\lhds x$---and this holds whenever $\mathcal N$ is \emph{converse}-strongly factoring.  Indeed, given $y\lhd x$, converse-strong factoring (applied to the pair $y\lhd x$, with $y$ in
the role of the first argument) yields $z$ with $z\lhds x$, $z\post y$,
and $y\pre z$, i.e., $y\prep z$ and $z\lhds x$.  Since the canonical
model is converse-strongly factoring (Theorem
\ref{thm:canonical-strongps}), the faithfulness argument below goes
through verbatim under the strengthened definition of FSTB-frames.
\end{remark}

\subsection{The main theorem for $\ITB$}

\begin{theorem}\label{thm:main-itb}
For all $\varphi,\psi\in\Fo$:
\[
\varphi\vdash_{\sysF}\psi
\iff
\varphi\tO\vdash_{\ITB}\psi\tO .
\]
\end{theorem}

\begin{proof}
($\Rightarrow$)  Suppose $\varphi\vdash_{\sysF}\psi$, and let
$\mathcal M$ be any FSTB-model.  By Theorem \ref{thm:F2}(1),
$F_2(\mathcal M)$ is an $\sysF$-model, so
$\den{\varphi}_{F_2(\mathcal M)}\subseteq\den{\psi}_{F_2(\mathcal M)}$
by soundness of $\sysF$ (Theorem \ref{thm:soundness}); by Theorem
\ref{thm:F2}(2), this says
$\den{\varphi\tO}_{\mathcal M}\subseteq\den{\psi\tO}_{\mathcal M}$.
Since $\mathcal M$ was arbitrary, completeness of $\ITB$ (right-to-left direction of Theorem
\ref{thm:itb-complete}) yields $\varphi\tO\vdash_{\ITB}\psi\tO$.

($\Leftarrow$)  Suppose $\varphi\nvdash_{\sysF}\psi$.  By Corollary
\ref{cor:strong-countermodel}, there are an $\sysF$-model $\mathcal M$
that is strongly factoring and a point $w$ with
$w\in\den{\varphi}_{\mathcal M}$ and
$w\notin\den{\psi}_{\mathcal M}$.  By Theorem \ref{thm:G2},
$G_2(\mathcal M)$ is an FSTB-model with
$\den{\varphi\tO}_{G_2(\mathcal M)}=\den{\varphi}_{\mathcal M}\ni w$ and
$\den{\psi\tO}_{G_2(\mathcal M)}=\den{\psi}_{\mathcal M}\not\ni w$.
Hence $\varphi\tO\nvdash_{\ITB}\psi\tO$ by soundness (left-to-right direction of Theorem
\ref{thm:itb-complete}).
\end{proof}

\begin{corollary}\label{cor:hm12}
For all $\varphi,\psi$ in the $\{\wedge,\vee,\neg\}$-fragment:
$\varphi\vdash_{\mathcal F}\psi$ iff
$\varphi\tO\vdash_{\ITB}\psi\tO$.  In view of Remark
\ref{rem:transO-neg}, this recovers Theorem 1.2 of \cite{HM2026}.
\end{corollary}

\begin{proof}
Combine Corollary \ref{cor:conservative} with Theorem
\ref{thm:main-itb}.
\end{proof}

Putting the two main theorems together, the fundamental logician with a
preconditional, the orthologician with the Sasaki hook and an S4
necessity, and the intuitionistic logician with a KTB modality can each
interpret the reasoning of the others:

\begin{corollary}\label{cor:threeway}
For all $\varphi,\psi\in\Fo$:
\[
\varphi\vdash_{\sysF}\psi
\iff
\varphi\tI\vdash_{\OS}\psi\tI
\iff
\varphi\tO\vdash_{\ITB}\psi\tO .
\]
\end{corollary}

\section{Conclusion and further directions}\label{sec:conclusion}

We have developed a propositional extension of fundamental logic by a
preconditional connective. On the proof-theoretic side, the systems
$\SysK\subseteq\SysT\subseteq\SysF$ successively axiomatize the
preconditional principles, conditional identity together with the
semicomplementation of the derived negation, and the additional
double-negation principle characteristic of fundamental logic. On
the semantic side, all three systems are matched with algebraic semantics,
while $\SysT$ and $\SysF$ receive relational semantics over openness frames,
strong completeness, and the finite model property. Finally, the modal translations
show that the resulting logic fits into the same triangle as fundamental
logic, ortho-$\mathsf{S4}$, and intuitionistic $\mathsf{KTB}$, with the
preconditional becoming a boxed Sasaki hook on the ortho-$\mathsf{S4}$ side
and, on the intuitionistic $\mathsf{KTB}$ side, a strict clause whose
classical $\mathsf{KTB}$ reading is equivalent to the Goldblatt image of that
hook.

Several natural questions remain. 
First, the finite model property gives
decidability but only a coarse upper bound; it would be desirable to find a
cut-free or analytic proof system, in the spirit of the sequent calculus of
fundamental logic in \cite{AB2023}, and to determine sharper complexity bounds. 
A further direction is to extend the present systems to quantified
fundamental logic with a preconditional. This would build directly on
\cite[\S\S~3.5 and~4.5]{HM2026}, where Holliday and Massas already extend both
modal translations to first-order fundamental logic. Their GMT-style
translation into $\mathsf{QOS4}$ uses
\[
\mu(\forall x\varphi)=\Box\forall x\mu(\varphi),
\qquad
\mu(\exists x\varphi)=\exists x\mu(\varphi),
\]
while their Goldblatt-style translation into $\mathsf{QFSTB}$ uses
\[
\gamma(\forall x\varphi)=\forall x\gamma(\varphi),
\qquad
\gamma(\exists x\varphi)=\Box\Diamond\exists x\gamma(\varphi);
\]
both are full and faithful
\cite[Lemmas~3.23 and~4.24 and Corollaries~3.25 and~4.26]{HM2026}. The new
task is therefore to add the preconditional to these quantified translations
and determine whether the truth-transfer and companion constructions for its
modal clauses commute with the interpretation of the quantifiers. This
requires analyzing how $\to$ interacts with the meets and joins interpreting
$\forall$ and $\exists$, as well as the role of the Barcan and constant-domain
principles noted in \cite[Remarks~3.26 and~4.27]{HM2026}. 


\end{document}